\documentclass[10pt,reqno]{amsart}
\usepackage{a4wide}

\usepackage{amsmath,amsfonts,amsthm,amssymb,amsxtra,dsfont}
\usepackage{amsxtra, amssymb, mathrsfs, pstricks}
\usepackage[english]{babel}
\usepackage{amsmath,amsfonts,amstext,amsbsy,amsopn,amssymb,amsthm,amscd}
\usepackage{amsxtra,latexsym}
\usepackage{exscale}
\usepackage{graphicx,mathrsfs}
\usepackage{psfrag}
\usepackage{paralist,eucal,enumerate}
\usepackage{color,graphics}

\usepackage{subcaption}

\newtheorem{theorem}{Theorem}[section]
\newtheorem{proposition}[theorem]{Proposition}

\newtheorem{lemma}[theorem]{Lemma}

\theoremstyle{definition}

\newtheorem{definition}[theorem]{Definition}

\theoremstyle{remark}
\newtheorem{remark}[theorem]{\bf Remark}

\numberwithin{equation}{section}

\usepackage{graphicx,mathrsfs}
\newcommand{\dd}{\mathrm{d}}

\newcommand{\field}[1]{\mathbb{#1}}
\newcommand{\N}{\field{N}}
\newcommand{\R}{\field{R}}

\newcommand{\Z}{\field{Z}}

  \def\bbR{{\mathbb R}}
  
\def\bbV{{\mathbb V}}  
 
 \def\calE{{\mathcal E}} \def\calF{{\mathcal F}}
\def\calG{{\mathcal G}}  \def\calI{{\mathcal I}}
 \def\calK{{\mathcal K}} \def\calL{{\mathcal L}}
  
  \def\calR{{\mathcal R}}
  
\def\calV{{\mathcal V}}  \def\calX{{\mathcal X}}
 \def\calZ{{\mathcal Z}}
   
\def\rmd{{\mathrm d}}

\def\rmm{{\mathrm m}}   
\def\rmp{{\mathrm p}}

\def\rmD{{\mathrm D}}

\newcommand{\wt}{\widetilde}

 \DeclareMathOperator{\dist}{dist}

\DeclareMathOperator{\diam}{diam}
\DeclareMathOperator{\Lin}{Lin}
\DeclareMathOperator{\Var}{Var}
 \renewcommand{\d}{{\mathrm d}}
 
 \newcommand{\ds}{\,\d s}
 \newcommand{\dt}{\,\d t}

 \newcommand{\dr}{\,\d r}

\newcommand{\abs}[1]{\left\lvert#1\right\rvert}      

\newcommand{\babs}[1]{\bigl\vert#1\bigr\vert} 

 \newcommand{\norm}[1]{\left\lVert#1\right\rVert}

\newcommand{\bnorm}[1]{\bigl\Vert#1\bigr\Vert} 

\newcommand{\Set}[2]{\{\,#1\,;\,#2\,\}}

\newcommand{\vvp}{\rmp}
\newcommand{\vvm}{\rmm} %

\definecolor{darkgreen}{rgb}{0.0, 0.5, 0.0}

\begin{document}

\title[ Optimal control and BV-solutions]{Optimal control of a rate-independent 
system constrained to parametrized balanced 
viscosity solutions}

\author{Dorothee Knees}
\address{Dorothee Knees, Institute of Mathematics, University of Kassel,
Heinrich-Plett Str.~40, 34132 Kassel, Germany. Phone:
+49 0561 8044355}
\email{dknees@mathematik.uni-kassel.de}

\author{Stephanie Thomas}
\address{Stephanie Thomas, Institute of Mathematics, University of Kassel,
Heinrich-Plett Str.~40, 34132 Kassel, Germany. Phone:
+49 0561 8044376
}
\email{sthomas@mathematik.uni-kassel.de}

\date{}

\begin{abstract}
We analyze an optimal control problem governed by a rate-independent system in 
an abstract infinite-dimensional setting. The rate-independent system is 
characterized by a  nonconvex stored energy functional, which 
depends  on time via a time-dependent external loading, 
and by a convex dissipation potential, which is assumed to be bounded and 
positively homogeneous of degree one.  

The optimal control problem uses the external load as control variable and is 
constrained to normalized parametrized balanced viscosity solutions (BV 
solutions) of the rate-independent system. Solutions of this type appear as 
vanishing viscosity limits  of viscously regularized versions of the original 
rate-independent system. Since BV solutions in general are not unique, as a 
main ingredient for the existence of optimal solutions we prove the compactness 
of solution sets for BV solutions.
\end{abstract}



\maketitle


\setcounter{tocdepth}{3}
{\small \tableofcontents}
\section{Introduction}
\label{s:introduction}
In this paper, we focus on the optimal control of a 
\textit{rate-independent system}. This system is given in terms of a 
\textit{state variable} $z:[0,T]\to\calZ$, a time-dependent \textit{external 
load} $\ell$, a \textit{stored energy functional} $\calE$ depending on $\ell$ 
and $z$, and a \textit{dissipation potential} $\calR:\calZ\to[0,\infty)$, which 
captures the dissipation due to internal friction. To be more precise, we assume 
that the state space $\calZ$ is a separable Hilbert space which fulfills the 
embedding $\calZ\Subset\calV\subset\calX$ for another separable Hilbert space 
$\calV$ and a Banach space $\calX$  
 and choose $\ell\in H^1(0,T;\calV^*)$.
We are working with a \textit{semilinear model}, i.e., we assume that there are 
a linear operator $A\in\mathrm{Lin}(\calZ,\calZ^*)$ and a nonlinearity 
$\calF:\calZ\to[0,\infty)$ such that $\calE:[0,T]\times \calZ\to\R$ is given by
\begin{align*}
\calE(t,z):=\tfrac12\langle Az,z\rangle_{\calZ^*,\calZ}+\calF(z)-\langle 
\ell(t),z\rangle_{\calV^*,\calV} =\calI(z)-\langle 
\ell(t),z\rangle_{\calV^*,\calV},
\end{align*}
where
\begin{align*}
\calI:\calZ\to\R,\quad \calI(z):=\tfrac12\langle Az,z\rangle_{\calZ^*,\calZ}+\calF(z)
\end{align*}
depends solely on the state $z$. 
$\calF$ is supposed to be \textit{nonconvex} and of lower order with 
respect to $A$. The precise assumptions on $A$, $\calF$ and $\ell$ can be 
found in section \ref{sec:basicassumptions}. Rate-independence means that the 
system is invariant w.r.t. time rescaling in the sense that, given a time 
rescaling, the solutions of the rescaled system are exactly the rescaled 
solutions of the original system. In order to obtain rate-independence, the 
dissipation potential $\calR$ is assumed not only to be continuous and convex, 
but also \textit{positively homogeneous of degree one}. In this paper, we are 
dealing with a \textit{bounded dissipation potential}, meaning that we also 
assume that there are constants $c,C>0$ such that
\begin{align}
\text{for all } z\in\calX:\quad c\Vert z\Vert_\calX\leq\calR(z)\leq C\Vert z\Vert_\calX.\label{eq.Mief100i}
\end{align}
 With these ingredients, the evolution of the state variable $z$ can be described by means of the \textit{doubly nonlinear equation}
\begin{align}
0\in\partial\calR(\dot z(t))+\rmD_z\calE(t,z(t)) \quad \text{ for a.a. } t\in[0,T],\label{def:rate-indep-sys}
\end{align}
where $\rmD_z\calE$ is the G\^{a}teaux derivative of $\calE$ w.r.t. $z$ and $\partial\calR:\calZ\rightrightarrows\calZ^*$ denotes the convex subdifferential of $\calR$. The aim of the paper is to show  existence of a globally optimal solution of an optimal control problem of the type
\begin{align}
 \left. \begin{array}{l} \text{min}\quad \Vert\hat z-z_\text{des}\Vert+\alpha\Vert\ell\Vert_{H^1(0,T;\calV^\ast)}\\ \text{s.t. }\quad \hat z\in \wt M_{ad},\end{array} \right\}\label{vage:optcontri}
\end{align}
where the external load $\ell$ is the control variable, $\alpha>0$ is a fixed 
Tikhonov parameter, and $z_\text{des}$ is a given desired state. We restrain the 
problem to an admissible set $\wt M_{ad}$ consisting of all solutions of 
\eqref{def:rate-indep-sys} in the sense of so-called \textit{parametrized BV 
solutions}. 

It is well known that rate-independent systems with nonconvex energy 
$\calE$ in general do not admit solutions that are continuous in time. Several 
solution concepts are available in the literature that allow for discontinuous 
solutions. We mention here the meanwhile classical  \textit{global energetic 
solutions} (GES) first proposed in \cite{MR2210284,MielkeTheilLevitas02} 
and \textit{balanced viscosity solutions} (BV solutions) that were first 
discussed in \cite{MiEf06} 
in a finite dimensional setting and later refined for instance in 
\cite{MRS16}. Due to a global stability criterion, GES  
tend to jump as early as possible even though a local stability criterion might 
predict a different behavior. In contrast to that, BV solutions tend to jump as 
late as possible.  
We refer to \cite{MieRou15} 
for more details and an overview on further solution concepts. Independently of 
the chosen solution concept, solutions of rate-independent systems with 
nonconvex energies in general are not unique. This is a major challenge when it 
comes to optimal control of such systems. 

The literature concerning the optimal control of rate-independent systems with 
nonconvex energies formulated on infinite dimensional spaces is rather scant. 
We mention here \cite{Rindler09}, where the existence of optimal solutions to 
a variant of the problem \eqref{vage:optcontri} constrained to global energetic 
solutions is shown. In \cite{MielkeRindler09,Rindler09b}, the  authors proved a 
reverse approximation property for global energetic solutions via time 
incremental solutions and used this property to show that global minimizers of 
optimal control problems governed by GES can be approximated by solutions of 
special time discrete optimal control problems.  To the best of our knowledge, 
no existence results are available in the literature for optimal control 
problems constrained to BV solutions. 

The basic solution concept in this paper are normalized, $\vvp$-parametrized 
balanced viscosity solutions, see Definition \ref{def:paramsol}. Thereby, the 
solutions are represented with respect to an artificial arc length parameter 
as in \cite{MiEf06} or in \cite[Definition 4.2]{MRS16}. 
Let us go into more details concerning the type of solutions and the results of 
this paper.


The existence of BV solutions can be shown via a vanishing viscosity 
approach.
Namely, the equation 
\eqref{def:rate-indep-sys} is approximated by a sequence of equations
\begin{align}
\label{intro.viscsystem}
0\in\partial\calR(\dot z_\varepsilon(t)) + \varepsilon\bbV\dot 
z_\varepsilon(t)
+\rmD_z\calE(t,z_\varepsilon(t)) \quad \text{ for a.a. } 
t\in[0,T], 
\end{align}
where  $\bbV\in\text{Lin}(\calV,\calV^*)$ is an elliptic and symmetric 
operator.
These types of \textit{viscous} systems have been analyzed in the past (see, 
e.g., \cite{MR3016509}) and are known to have absolutely continuous 
solutions $z_\varepsilon\in W^{1,1}(0,T;\calV)$. 
In order to identify the limit as the viscosity $\varepsilon$ tends to 
zero, one option is to reformulate the viscous system with respect to an 
artificial arc length parameter so that the trajectory 
$t\mapsto(t,z_\varepsilon(t))$ is rewritten as $s\mapsto(\hat 
t_\varepsilon(s),\hat z_\varepsilon(s))$. There are several possibilities for 
choosing the reparametrization. For our purpose, the reparametrization based on  
the so-called vanishing viscosity contact potential $\vvp(\cdot,\cdot)$ is the 
most appropriate one, \cite{MRS16}. Here, one sets 
\begin{align}
\label{eq.def.vanvisccontpot}
  s_\varepsilon(t):= t + \int_0^t\vvp(\dot 
z_\varepsilon(\tau),-\rmD\calE(\tau,z_\varepsilon(\tau)))\d\tau \text{ with }
\vvp(v,\xi):= \calR(v) 
+ 
\norm{v}_\bbV\dist_\bbV(\xi,\partial\calR(0))
\end{align}
and chooses $\hat t_\varepsilon$ as the inverse function of $s_\varepsilon$.  
Thereby, $\Vert\cdot\Vert_\bbV$ denotes the 
equivalent norm induced by $\bbV$ on $\calV$ and 
$\mathrm{dist}_\bbV(\eta,\partial\calR(0))$ is the distance of an element 
$\eta\in\calV^*$ to the subdifferential $\partial\calR(0)$ measured by the 
corresponding norm on the dual space $\calV^*$.
 Defining 
$\hat z_\varepsilon=z_\varepsilon\circ \hat 
t_\varepsilon:[0,S_\varepsilon]\to\calZ$, it is then possible to pass to the 
limit for vanishing viscosity (i.e., for $\varepsilon\to0$) and obtain limits 
$S\in[0,\infty)$ of $S_\varepsilon$, $\hat z \in \mathrm{AC}(0,S;\calX)$ of 
$\hat z_\varepsilon$ and $\hat t\in W^{1,\infty}(0,S;\R)$ of $\hat 
t_\varepsilon$. Simultaneously passing to the limit in the reparametrized energy 
dissipation balance associated with \eqref{intro.viscsystem}, 
one also obtains the energy dissipation 
balance fulfilled by $(\hat t,\hat z)$, which reads
\begin{align}
 \label{intro.eq.energydissipidentitylimitpparam}
 \calE(\hat t(s),\hat z(s)) + \int_0^s\calR[\hat z'](r) \dr + 
 \int_{(0,s)\cap G} \norm{\hat z'(r)}_\bbV\dist_\bbV(-\rmD \calE(\hat t(r),\hat 
z(s)),\partial\calR(0))\dr 
\\
= \calE(0,z_0) - \int_0^s\langle \ell'(\hat t(r))\hat t^\prime(r),\hat z(r)\rangle\dr\,.
\end{align}
 Here, it is possible to show that 
$\hat z\in\text{AC}_\text{loc}(G;\calV)$ is differentiable almost everywhere on 
the set \begin{align*}G=\lbrace s\in[0,S]\,|\, \dist_\bbV(-\rmD \calE(\hat 
t(s),\hat z(s)),\partial\calR(0))>0 \rbrace,\end{align*}
so that the second integrand is defined almost everywhere.
Normalized, $\vvp$-parametrized BV solutions then are defined as triples 
$(S,\hat t,\hat z)$ with certain regularities that satisfy  the 
energy dissipation identity \eqref{intro.eq.energydissipidentitylimitpparam} 
and that are normalized in the sense of \eqref{eq.normalized.sol} in Section 
\ref{sec:paramBVsol}. 
One advantage of using the parametrization \eqref{eq.def.vanvisccontpot} 
is that limits of solutions $(\hat t_\varepsilon,\hat 
z_\varepsilon)_\varepsilon$ are automatically normalized, a property that we 
will also exploit in the analysis for the optimal control problem. 
Since the focus of this paper is on the optimal control problem, we do not 
include a detailed proof of 
existence of parametrized BV solutions.

As already mentioned, BV solutions typically are not unique. Hence, 
for the purpose of optimal control one needs to show the sequential 
closedness of the graph of the set-valued solution operator and a compactness 
property. This is the contents of   Theorem 
\ref{thm.properties_sol_set}.  
For the proof  of  Theorem 
\ref{thm.properties_sol_set} the main challenge will be to derive a priori 
estimates for the driving forces $\rmD\calE(\hat t,\hat z)$ on the set $G$. In 
order to obtain these, we first show that for each parametrized BV solution, 
there exists a Lagrange parameter $\lambda:(0,S)\to[0,\infty)$ with 
$\lambda(s)=0$ on $(0,S)\setminus G$ such that the inclusion
\begin{align}
0\in\partial\calR(\hat z^\prime(s))+\lambda(s)\bbV\hat 
z^\prime(s)+\rmD\calI(\hat z(s))-\ell(\hat t(s))
\label{def:diffincllambdai} 
\end{align}
is fulfilled almost everywhere on $G$. For each connected component of 
$G$, we subsequently choose a reparametrization in such a way that the 
transformed functions are solutions of the system 
\begin{align}
0\in\partial\calR(\dot z(t))+\bbV\dot z(t)+\rmD\calI(z(t))-\ell_*
\quad \text{ for }t>0\label{eq:viscsysRplusi} 
\end{align}
for a constant load $\ell_*\in\calV^*$. We then prove existence and a priori 
estimates for solutions of \eqref{eq:viscsysRplusi} by formally differentiating 
the inclusion \eqref{eq:viscsysRplusi} w.r.t. $t$ and testing by $\dot z$. This 
is made rigorous relying on a regularization approach. Namely, for $\delta>0$, 
we analyze the system 
\begin{align*}
0\in\partial\calR_\delta(\dot 
z_\delta(t))+\rmD\calI(z_\delta(t))-\ell_* \quad \text{ for 
}t>0
\end{align*}
with the augmented dissipation potential
\begin{align*}
\calR_\delta(z)=\calR(z)+\tfrac12\langle \bbV 
z,z\rangle_{\calV^*,\calV}+\tfrac\delta2\langle Az,z\rangle_{\calZ^*,\calZ}. 
\end{align*}
This can be done relying mainly on ODE-arguments. We finally  transfer the a 
priori estimates thus obtained for \eqref{eq:viscsysRplusi} to the original 
system \eqref{def:diffincllambdai} by means of a change of variable. These 
essential estimates for parametrized BV solutions then allow us to show 
compactness of solution sets of the type
\begin{align*}
M_\varrho:=&\left\{(S,\hat t,\hat z)\,|\right.\\
&\hspace{2mm}\left.(S,\hat t,\hat z)\text{ is a parametrized BV solution for 
$(z_0,\ell)$ with } \norm{z_0}_\calZ + \norm{\ell}_{H^1((0,T);\calV^*)}\leq 
\varrho\right\},
\end{align*}
see Theorem \ref{thm.properties_sol_set}.

In the final section of the paper, we turn to an  \textit{optimal control 
problem governed by \eqref{def:rate-indep-sys}}, which is constrained to the 
admissible set
\begin{align*}
M_{ad}:=&\left\lbrace (S,\hat t,\hat z,\ell)\,|\, (S,\hat t,\hat z) \text{ is a parametrized BV solution for $(z_0,\ell)$} \right\rbrace.
\end{align*}
We are then in the position to prove an existence result for an optimal control problem of the type \eqref{vage:optcontri}, which now reads
\begin{align}
 \left. \begin{array}{l} \text{min}\quad J(S,\hat z,\ell):=j(\hat z(S))+\alpha\Vert\ell\Vert_{H^1(0,T;\calV^\ast)}\\ \text{s.t. }\quad (S,\hat t, \hat z,\ell)\in M_{ad}.\end{array} \right\}\label{def:optcontri}
\end{align}
Here, $\alpha>0$ is again a fixed Tikhonov parameter and $j:\calV\to \R$ is bounded from below and continuous , e.g. $j(z):=\Vert z-z_\text{des}\Vert_\calV$ for a desired end state $z_\text{des}\in\calV$.

\textbf{Plan of the paper:} In Section \ref{sec:basicassumptions}, we list basic 
assumptions and estimates for the energy functional $\calE$ and the dissipation 
potential $\calR$. In Section \ref{sec:paramBVsol}, we then give a definition 
and cite an existence result for parametrized BV solutions. We 
further provide  basic properties of  solutions, like for example 
the differential inclusion \eqref{def:diffincllambdai}. We next derive 
uniform estimates for the driving forces $\rmD\calE$ by analyzing the system 
\eqref{eq:viscsysRplusi} and transferring the results to 
\eqref{def:diffincllambdai} by means of a rescaling argument. The section closes 
with the compactness result for the sets $M_\varrho$. The paper is concluded in 
Section \ref{sec:optcontr} with the existence result for the optimal control 
problem \eqref{def:optcontri}. In the Appendix, we collect convergence 
results for the load term, lower semicontinuity properties of some 
functionals,  results for Banach space valued absolutely 
continuous functions, a combined Helly and Ascoli-Arzel\`a theorem, and a chain 
rule.

\section{Basic assumptions and estimates} 
\label{sec:basicassumptions}
The analysis will be carried out for the semilinear system introduced in 
\cite{MielkeZelik_ASNPCS14}, compare also   \cite[Example 
3.8.4]{MieRou15} and \cite{Knees2018a}. 

Let $\calX$ be a Banach 
space and $\calZ,\calV$ be separable Hilbert spaces that are densely and 
compactly resp.\ continuously embedded in the following way:
\begin{align}
\label{eq.Mief000}
 \calZ\Subset\calV\subset \calX.
\end{align}
Let further $A\in \Lin(\calZ,\calZ^*)$ and $\bbV\in \Lin(\calV,\calV^*)$ be 
linear symmetric, bounded  $\calZ$- and $\calV$-elliptic operators, i.e.\ there 
exist constants $\alpha,\gamma>0$ such that 
\begin{align}
\label{eq.Mief0001}
 \forall z\in \calZ, \forall v\in \calV:\quad 
 \langle Az,z\rangle\geq \alpha\norm{z}^2_\calZ\,,\quad 
\langle \bbV v,v\rangle\geq \gamma\norm{v}^2_\calV\,,
\end{align}
and $\langle Az_1,z_2\rangle =\langle A z_2,z_1\rangle$ for all $z_1,z_2\in 
\calZ$ (and similar for $\bbV$). 
Here, $\langle \cdot,\cdot\rangle$ denotes the duality pairings in $\calZ$ and 
$\calV$, respectively. We define $\norm{v}_\bbV:=\left(\langle \bbV 
v,v\rangle\right)^\frac{1}{2}$, which is a norm that is equivalent to the 
Hilbert space norm $\norm{\cdot}_\calV$. 
Let further 
\begin{align}
\label{eq.Mief00a}
  \calF\in C^2(\calZ;\R)\text{  with } 
\calF\geq 0.
\end{align} 
The functional $\calF$ shall 
play the role of a possibly nonconvex lower order term (cf.\ \cite[Section 
3.8]{MieRou15}). Hence, we assume that  
\begin{align}
\label{ass.F01}
 \rmD_z\calF\in C^1(\calZ;\calV^*),\quad \norm{\rmD^2_z\calF(z)v}_{\calV^*}\leq 
C(1 + \norm{z}_\calZ^q)\norm{v}_\calZ
\end{align}
for some $q\geq 1$. Let 
$ \ell\in H^1((0,T);\calV^*)$. 
Energy functionals of the following type are considered 
\begin{align}
\label{eq.Mief0002}
 \calI:&\calZ \rightarrow \R,\quad \calI(z):=\frac{1}{2}\langle A 
z,z\rangle + 
\calF(z) .
\\
\calE:&[0,T]\times\calZ\to\R, \quad \calE(t,z)=\calI(z) - \langle 
\ell(t),z\rangle\,.
\end{align}
Clearly, $\calI\in C^1(\calZ;\R)$. 
%
%
If not otherwise stated, in the whole paper we assume that the initial datum 
$z_0\in \calZ$ and the load $\ell$ are compatible in the following sense
\begin{align}
 \label{eq.compatibledata}
 z_0\in \calZ,\, \ell\in H^1((0,T);\calV^*) \text{ and }
 \rmD\calE(0,z_0)=\rmD\calI(z_0) - \ell(0) \in \calV^*.
\end{align}

The dissipation functional  $\calR:\calX\to[0,\infty)$ is assumed to be  
convex, continuous, positively homogeneous 
of degree one and 
\begin{align}
\label{eq.Mief100}
 \exists c,C>0\, \forall x\in \calX:\quad c\norm{x}_\calX\leq \calR(x)\leq 
C\norm{x}_\calX\,.
\end{align}
%
%
>From \eqref{ass.F01} and \eqref{eq.Mief100} we deduce the following 
interpolation estimate, \cite[Lemma 1.1]{Knees2018a}: 
\begin{lemma}
 \label{lem.estDF}
 Assume \eqref{eq.Mief000}, \eqref{eq.Mief00a}, \eqref{eq.Mief100} and 
\eqref{ass.F01}. For every 
$\rho>0$ and $\kappa>0$ there exists $C_{\rho,\kappa}>0$ such that 
for all $z_1,z_2\in \calZ$ with $\norm{z_i}_\calZ\leq \rho$ we have 
\begin{align}
 \label{est:DF}
 \abs{\langle \rmD\calF(z_1)-\rmD\calF(z_2),z_1 - z_2\rangle}\leq \kappa 
\norm{z_1-z_2}^2_\calZ + 
C_{\rho,\kappa}\min\{\calR(z_1-z_2),\calR(z_2 - z_1)\}\norm{z_1-z_2}_\bbV.
\end{align} 
\end{lemma}
As a consequence, $\calE$ is  $\lambda$-convex on sublevels. To be more 
precise, we have the following estimate: For every $\rho>0$ there exists 
$\lambda=\lambda(\rho)>0$ such that for all $t\in [0,T]$ and all $z_1,z_2\in 
\calZ$ with $\norm{z_i}_\calZ\leq \rho$ we have
\begin{align}
 \label{est.lambda-convex}
 \langle \rmD_z\calE(t,z_1) - \rmD_z\calE(t,z_2), z_1 - 
z_2\rangle_{\calZ^*,\calZ}\geq \tfrac{\alpha}{2}\norm{z_1 - z_2}_\calZ^2 - 
\lambda\norm{z_1 - z_2}_\bbV^2 
\end{align}
and
\begin{align}
 \label{est.lambda-convex-energy}
 \calI(z_2) - \calI(z_1) \geq \langle \rmD_z\calI(z_1), z_2 - 
z_1\rangle_{\calZ^*,\calZ} + \tfrac{\alpha}{2}\norm{z_1 - z_2}^2_\calZ
- \lambda \calR(z_2 -z_1)\norm{z_2 - z_1}_\calV.
\end{align}
%

Finally, we assume that  
\begin{align}
\label{ass.fweakconv}
 \calF:\calZ\rightarrow\R \text{ and }\rmD_z\calF:\calZ\rightarrow\calZ^* 
\text{ are weak-weak continuous.}
\end{align}

\section{Parametrized BV-solutions and properties of the solution set}
\label{sec:paramBVsol}

\subsection{Definition of parametrized balanced viscosity solutions}
We use the following notation: for $\xi\in \calV^*$, 
\begin{align*}
\dist_\bbV(\xi,\partial\calR(0)):= 
\inf\Set{\norm{\xi-\sigma}_{\bbV^{-1}}}{\sigma\in \partial\calR(0)},
\end{align*}
where $\norm{\xi}_{\bbV^{-1}}^2:=\langle\xi,\bbV^{-1}\xi\rangle$. 
Moreover, for 
$\ell\in \calV^*$, $z\in \calZ$ we define
\begin{align}
\label{def.vvm}
\vvm(\ell,z):=\dist_{\bbV}(-\rmD\calI(z)+\ell,\partial\calR(0)).
\end{align}

\begin{definition}
 \label{def:paramsol} 
 Let $z_0\in \calZ$ and $\ell\in H^1((0,T);\calV^*)$. 
 A triple $(S,\hat t,\hat z)$ with 
 $S>0$, $\hat t\in W^{1,\infty}((0,S);\R)$, $\hat z\in 
AC^\infty([0,S];\calX)\cap L^\infty((0,S);\calZ)$ 
  is a normalized, $\vvp$-parametrized balanced 
viscosity solution  of the 
rate-independent system $(\calI,\calR)$ with data  
$z_0,\ell$,
if there exists a (relatively) open set 
$G\subset[0,S]$ such that $\hat z\in W^{1,1}_{\text{loc}}(G;\calV)$, 
   $\rmD\hat\calE(\cdot,\hat z(\cdot))\in L^\infty_\text{loc}(G;\calV^*)$  
and such 
that    
$\vvm(\hat\ell(s),\hat z(s))>0$ on $G$  and $\vvm(\hat\ell(s),\hat 
z(s))=0$ on $[0,S]\backslash G$. 
Let $\hat \ell:=\ell\circ \hat t$ and $\hat 
\calE(s,v):=\calI(v)-\langle\hat\ell(s),v\rangle$. In addition to 
the above,  the following relations shall be satisfied:
\\
\textit{Complementarity and normalization condition:} 
For almost every $s\in [0,S]$ 
\begin{align}
\hat t'(s)\geq 0,\,\, \hat t(S)=T,\,\, \hat z(0)=z_0,
\label{eq.pparamsol-con1}
\\
  \hat t'(s)\dist_\bbV(-\rmD\hat\calE(s,\hat 
z(s)),\partial\calR(0))=0\,,
\label{eq.complementarity}
\\
1= \begin{cases}
 \hat t'(s) +\calR[\hat z'](s) &\text{if }s\notin G,\\
 \hat t'(s) +\calR[\hat z'](s) +\norm{\hat z'(s)}_\bbV\dist_\bbV(-\rmD\hat 
\calE(s,\hat z(s)),\partial\calR(0)) &\text{if }s\in G.
\end{cases}
\label{eq.normalized.sol}
  \end{align}
\textit{Energy-dissipation balance:} For every $s\in [0,S]$ 
\begin{multline}
 \label{eq.energydissipidentitylimitpparam}
 \hat\calE(s,\hat z(s)) + \int_0^s\calR[\hat z'](r) \dr + 
 \int_{(0,s)\cap G} \norm{\hat z'(r)}_\bbV\dist_\bbV(-\rmD\hat \calE(r,\hat 
z(r)),\partial\calR(0))\dr 
\\
= \hat\calE(0,z_0) - \int_0^s\langle \hat\ell'(r),\hat z(r)\rangle\dr\,.
\end{multline}

With $\calL(z_0,\ell)$ we denote the set of normalized, $\vvp$-parametrized 
balanced viscosity solutions  associated with the pair $(z_0,\ell)$.
\end{definition}
\begin{remark}
This definition is an adapted version of Definition 4.2 from 
\cite{MRS16}.
\end{remark}

\begin{theorem}
Under the conditions formulated in Section \ref{sec:basicassumptions}, for 
every compatible $z_0\in \calZ$ and $\ell\in H^1((0,T);\calV^*)$ (see 
\eqref{eq.compatibledata}), there exists at least one normalized, 
$\vvp$-parametrized balanced 
viscosity solution  of the rate-independent system $(\calI,\calR)$. In other 
words, $\calL(z_0,\ell)$ is not the empty set. 
\end{theorem}

Remarks on the proof: Assuming more regularity on $\ell$, namely $\ell\in 
C^1([0,T];\calV^*)$, this theorem is a special case of \cite{Mie11}, 
\cite[Theorem 4.3]{MRS16}, 
see also \cite{Knees2018a}, where  the situation of the present 
article 
is discussed. For the case $\ell\in BV([0,T];\calV^*)$ we refer to  
\cite{KnZa18prep}. A typical strategy to prove the existence of solutions is to 
follow a vanishing viscosity approach. This means that in a first step the 
existence of solutions of the viscously regularized systems
\begin{align}
\label{eq.viscreg}
 0\in \partial\calR(\dot z_\varepsilon(t)) +\varepsilon\bbV\dot 
z_\varepsilon(t) +\rmD\calE(t,z_\varepsilon(t)),\,\, z_\varepsilon(0)=z_0,\quad 
t\in (0,T), \,\,\varepsilon>0
\end{align}
is shown and a priori estimates are derived that are uniform with respect 
to the viscosity parameter $\varepsilon$. In a second step, the viscous 
solutions are reparametrized and the passage to the limit $\varepsilon\to 
0$ is carried out in the reparametrized setting. In order to obtain the above 
introduced parametrized solutions, one uses the change of variables
\[
 s_\varepsilon(t):= t + \int_0^t\vvp(\dot 
z_\varepsilon(\tau),-\rmD\calE(\tau,z_\varepsilon(\tau)))\d\tau \text{ with }
\vvp(v,\xi):= \calR(v) 
+ 
\norm{v}_\bbV\dist_\bbV(\xi,\partial\calR(0))
\]
and defines $\hat t_\varepsilon:=s_\varepsilon^{-1}$ as the inverse function and 
$\hat z_\varepsilon:=z_\varepsilon\circ \hat t_\varepsilon$. Reformulating the 
energy-dissipation balance associated with \eqref{eq.viscreg} in the new 
variables and passing to the limit in this expression ultimately leads to a 
normalized,  
$\vvp$-parametrized balanced 
viscosity solution  of the rate-independent system $(\calI,\calR)$ in the 
sense of Definition \ref{def:paramsol}. The 
quantity $\vvp(\cdot,\cdot)$ is the so-called \textit{vanishing viscosity 
contact potential}, \cite{MRS16}. 
Since the aim of this paper is to discuss an optimal control problem 
taking parametrized solutions as constraints, we do not go into further details 
concerning the existence of solutions.

\subsection{Alternative representation of BV solutions and basic uniform 
estimates}

\begin{proposition}
 \label{prop.basicprop.paramsol}
Assume \eqref{eq.compatibledata}.

Every normalized, $\vvp$-parametrized 
balanced viscosity solution $(S,\hat t,\hat z)\in \calL(z_0,\ell)$  of the 
rate-independent system satisfies 
\begin{enumerate}
\item The mapping $s\mapsto \calE(s,\hat z(s))$ belongs to 
$AC^2([0,S];\R)$.

\item 
 $\hat t$ is constant on the closure of each connected component of $G$ and 
there 
exists a measurable function $\lambda:(0,S)\to[0,\infty)$ with 
$\lambda(s)=0$ on $(0,S)\backslash G$ such that on each connected 
component $(a,b)\subset G$ the differential inclusion
\begin{align} 
\label{eq.diffincllambda}
 0\in\partial\calR(\hat z'(s)) +\lambda(s)\bbV\hat z'(s) + 
\rmD\hat\calE(s,\hat z(s))
\end{align}
is satisfied, for almost all $s\in (a,b)$. 
\\
For almost all $s\in G$ we have 
$\lambda(s)=\dist_\bbV(-\rmD\hat\calE(s,\hat z(s)),\partial\calR(0))/\norm{\hat 
z'(s)}_\bbV$. 
 
\item Basic energy estimates: There exists a constant $c>0$ (depending on the 
ellipticity constant $\alpha$ in \eqref{eq.Mief0001} and embedding constants, 
only) such that for all $(z_0,\ell)$ with \eqref{eq.compatibledata} and all  
 $(S,\hat t,\hat z)\in \calL(z_0,\ell)$ it holds
 \begin{align}
 \label{eq.unif-est1-paramsola}
  \norm{\hat z}_{L^\infty(0,S;\calZ)}&\leq c(1 +\abs{\calE(0,z_0)} + 
\norm{\ell}_{W^{1,1}((0,T);\calV^*)}), 
\\
S=\int_0^S 
\calR[\hat z'](s)\ds &+ \int_{(0,S)\cap G}\norm{\hat 
z'(s)}_\bbV\dist_\bbV(-\rmD\hat \calE(s,\hat z(s)),\partial\calR(0))\ds
\nonumber
\\
&\leq c\big(1 +\abs{\calE(0,z_0)} + 
\norm{\ell}_{W^{1,1}((0,T);\calV^*)} \big)^2\,.
  \label{eq.unif-est1-paramsolb}
 \end{align}
\end{enumerate} 
\end{proposition}
\begin{proof} 
Claim (1): Taking into account the normalization condition 
\eqref{eq.normalized.sol},  from the energy dissipation balance 
\eqref{eq.energydissipidentitylimitpparam} we obtain for every $s<\sigma\in 
[0,S]$:
\begin{align*}
 \abs{\hat \calE(\sigma,\hat z(\sigma)) - \hat \calE(s,\hat z(s))}\leq 
\int_s^\sigma\left(\babs{\langle\hat\ell'(r),\hat z(r)\rangle} + 1\right)\dr.
\end{align*}
Thanks to Proposition \ref{prop.hellyarzasc-version2} we have  $\hat 
z\in C([0,S];\calV)$ and $\hat \ell\in H^1((0,S);\calV^*)$ thanks to 
Lemma \ref{lemma:convell}. Hence, the integrand  belongs to 
$L^2((0,S);\R)$ from which claim (1) ensues. 

 Claim (2) is a standard property of nondegenerate parametrized solutions, 
cf.\ \cite{Mie11} and we give the proof here for completeness. 
Since $m(\hat\ell(s),\hat 
z(s))>0$  on $G$, from the complementarity 
condition \eqref{eq.complementarity} we deduce that $\hat t$ is constant on 
each connected component of $G$. 
In order to verify \eqref{eq.diffincllambda}, let $[a,b]\Subset G$. Since by 
assumption $\hat z\in W^{1,1}((a,b);\calV)$ we have 
$\calR[\hat 
z'](s)=\calR(\hat z'(s))$  for almost all $s\in (a,b)$, cf.\  
 \cite[Remark 1.1.3]{AGS05}. 
Thus, localizing the energy dissipation identity 
\eqref{eq.energydissipidentitylimitpparam} 
(where we apply the chain rule formulated in Proposition  
\ref{prop.chainrules}) yields
\begin{align}
\label{eq.locengdissipid}
\calR(\hat z'(s)) + \langle\rmD\hat\calE(s,\hat z(s)),\hat 
z'(s)\rangle_{\calV^*,\calV} + \norm{\hat z'(s)}_\bbV\dist(-\rmD\hat 
\calE(s,\hat z(s)),\partial\calR(0)) =0
\end{align}
which is valid for almost all $s\in (a,b)$. Since $\hat t$ is constant 
on $(a,b)$, from  \eqref{eq.normalized.sol} it follows that $\hat z'(s)\neq 0$ 
almost everywhere on $(a,b)$. Hence, with 
\[
\lambda(s)= 
\begin{cases}
\dist_\bbV(-\rmD\hat\calE(s,\hat z(s)),\partial\calR(0))/\norm{\hat 
z'(s)}_\bbV,&\text{if }\hat z'(s)\neq 0,\\
0,&\text{otherwise}
\end{cases}
\]
we have $ \norm{\hat z'(s)}_\bbV\dist(-\rmD\hat 
\calE(s,\hat z(s)),\partial\calR(0)) =\langle \lambda(s)\bbV\hat z'(s),\hat 
z'(s)\rangle$ and \eqref{eq.diffincllambda} follows from 
\eqref{eq.locengdissipid} and the one-homogeneity of $\calR$. This finishes the 
proof of claim (2) in Lemma \ref{prop.basicprop.paramsol}.

Claim (3): The verification of 
\eqref{eq.unif-est1-paramsola}--\eqref{eq.unif-est1-paramsolb} takes the energy 
dissipation
estimate as a starting point. Indeed, for all $b\in [0,S]$ from the energy 
dissipation balance we obtain 
\begin{align*}
 \hat\calE(b,\hat z(b))\leq \calE(0,z_0) +c_\calZ 
\bnorm{\dot\ell}_{L^1((0,T);\calV^*)} 
 \norm{\hat z}_{L^\infty((0,S);\calZ)},
\end{align*}
where the constant $c_\calZ$ is related with the embedding $\calZ\subset\calV$. 
On the other hand, due to the structure of $\calE$, we have $\hat\calE(b,\hat 
z(b))\geq \frac{\alpha}{4}\norm{\hat z(b)}_\calZ^2 - 
c_\alpha c_\calZ\norm{\ell}_{L^\infty((0,T);\calV^*)}^2$. Combining these 
estimates yields \eqref{eq.unif-est1-paramsola}. Estimate 
\eqref{eq.unif-est1-paramsolb} now is immediate.
\end{proof}

\subsection{A uniform estimate for the driving 
forces $\rmD\calE$ and a viscous  model on $\R_+$}

By assumption, parametrized solutions $(S,\hat t,\hat z)$ satisfy 
$\rmD\hat\calE(\cdot,\hat z(\cdot))\in L^\infty_\text{loc}(G;\calV^*)$. 
Moreover, for $s\in [0,S]\backslash G$ we have $m(\hat\ell(s),\hat z(s))=0$ 
which implies that $-\rmD\hat \calE(s,\hat z(s))\in \partial\calR(0)$. Since 
$\partial\calR(0)$ is  a bounded subset of $\calV^*$ we obtain  
$\rmD\hat\calE(\cdot,\hat z(\cdot))\in L^\infty((0,S)\backslash G;\calV^*)$. 
The goal of this section is to show that $\rmD\hat\calE(\cdot,\hat z(\cdot))\in 
L^\infty((0,S);\calV^*)$ and to derive estimates that are uniform on sets of 
the type 
\begin{align*}
M_\varrho:=\Set{(S,\hat t,\hat z)}{(S,\hat t,\hat z)\in 
\calL(z_0,\ell) \text{ for $(z_0,\ell)$ with 
\eqref{eq.compatibledata} and }\norm{z_0}_\calZ + 
\norm{\ell}_{H^1((0,T);\calV^*)}\leq \varrho}.
\end{align*}
For the proof of such uniform estimates we consider the inclusion 
\eqref{eq.diffincllambda} on connected components of $G$ and reparametrize it 
in such a way that the transformed function $\wt z$ satisfies 
\begin{align*}
 0\in \partial\calR(\wt z'(r)) + \bbV\wt z'(r) +\rmD\calI(\wt z(r))-\ell_*, 
\,\, r>0 
\end{align*}
with a constant load $\ell_*$. 
The essential estimates will be derived for this system and 
subsequently transferred to the original one. 

\subsubsection{An autonomous viscously regularized  rate-independent system on 
$\R_+$} 
The aim of this section is to derive regularity properties and estimates for 
solutions of the system
\begin{align}
\label{eq.viscsysRplus}
 0\in \partial\calR(\dot z(t)) +\bbV\dot z(t) +\rmD\calI(z(t)) -\ell_*,\,\, t>0
\end{align}
being defined on $\R_+=(0,\infty)$ with a constant load $\ell_*\in 
\calV^*$.

\begin{theorem}
\label{thm.viscRplusexregest}
\begin{enumerate} 
\item Uniqueness of solutions: For every $\ell_*\in \calV^*$ 
and $z_0\in \calZ$  there exists at most one 
function $z\in  L^\infty(\R_+;\calZ)$ with $\dot z\in L^1((0,a);\calV)$ for 
every $a>0$  that 
satisfies 
$z(0)=z_0$ and the inclusion \eqref{eq.viscsysRplus} for almost all $t>0$.  
  \item Existence of solutions and regularity: For every $\ell_*\in \calV^*$ 
and $z_0\in \calZ$ with $\rmD\calI(z_0)\in \calV^*$ there exists a unique 
function $z\in  L^\infty(\R_+;\calZ)$ with $\dot z\in L^2(\R_+;\calV)$ that 
satisfies 
$z(0)=z_0$ and the inclusion \eqref{eq.viscsysRplus} for almost all $t>0$.  
Moreover, this solution 
belongs to $W^{1,\infty}(\R_+;\calV)$ with 
$\Var_\calZ(z;[0,\infty))<\infty$ and 
$\rmD\calI(z(\cdot))\in L^\infty(\R_+;\calV^*)$. 
\item Uniform estimates: There exist functions 
$m_1,m_2:\calZ\times\calV^*\to[0,\infty)$ that map bounded sets on bounded sets 
such that for all $\ell_*\in \calV$ and  all $z_0\in \calZ$ with 
$\rmD\calI(z_0)\in \calV^*$ it holds: let $z$ be the solution of 
\eqref{eq.viscsysRplus} corresponding to $(z_0,\ell_*)$. Then 
\begin{align}
 \norm{z}_{L^\infty(\R_+;\calZ)}&\leq 
 m_1(z_0,\ell_*),
\label{est.viscRp1}\\
\norm{\dot z}_{L^\infty(\R_+;\calV)} + \Var_\calZ(z;[0,\infty))
&\leq 
m_2(z_0,\ell_*)
\big(\dist_\bbV(-\rmD\calI(z_0)+\ell_* ,\partial\calR(0)) 
+  m_1(z_0,\ell_*)\big),
\label{est.viscRp2}\\
\norm{\rmD\calI(z(\cdot))}_{L^\infty(\R_+;\calV^*)}& \leq 
\diam_{\calV^*}(\partial\calR(0)) +\norm{\ell_*}_{\calV^*}
+ c_\bbV \norm{\dot z}_{L^\infty(\R_+;\calV)} \,.
\label{est.viscRp3}
\end{align}
\end{enumerate}
\end{theorem}
\begin{remark}
\label{rem.constsol}
 Let $z_0\in \calZ$, $\ell_*\in \calV^*$ and assume that $-\rmD\calI(z_0) + 
\ell_*\in \partial\calR(0)$. Then the constant function $z(t)=z_0$, $t>0$, is 
the unique solution of \eqref{eq.viscsysRplus}. If 
$-\rmD\calI(z_0)+\ell_*\notin \partial\calR(0)$, then along the whole solution 
curve 
we have $-\rmD\calI(z(t))+\ell_*\notin \partial\calR(0)$. 
\end{remark}

 For deriving the uniform estimates 
\eqref{est.viscRp2}--\eqref{est.viscRp3} one formally takes the derivative of 
the inclusion \eqref{eq.viscsysRplus} with respect to $t$ and chooses $\dot z$ 
as a test function.  This can be made rigorous on a time-discrete level or 
alternatively by an argument relying on a regularized version of 
\eqref{eq.viscsysRplus}. In the presentation here, we choose the latter 
approach. 

For $\delta>0$ and $v\in \calZ$ let $\calR_\delta(v):=\calR_\bbV(v) 
+\frac{\delta}{2}\langle A v,v\rangle_{\calZ^*,\calZ}$ with 
$\calR_\bbV(v)=\calR(v) +\tfrac{1}{2}\langle \bbV v,v\rangle$ and with the 
operator $A$ from \eqref{eq.Mief0001} (any operator of that type 
would do).

\begin{proposition}
 \label{prop.exregularizedRplus}
 For every $\delta>0$ and every $z_0\in \calZ$, $\ell_*\in\calV^*$ there exists 
a unique function $z_\delta\in W^{2,\infty}(\R_+;\calZ)$ satisfying 
$z_\delta(0)=z_0$ and 
\begin{align}
\label{eq.viscdelta}
 0\in\partial\calR_\delta(\dot z_\delta(t)) +\rmD\calI(z_\delta(t))-\ell_*.
\end{align}
Moreover, $z_\delta$ fulfills the energy-dissipation balance
\begin{align}
\label{eq.edb-delta}
 \calE(t,z_\delta(t))+\int_0^t \calR_\bbV(\dot z_\delta(\tau)) 
+\tfrac{\delta}{2}\langle A\dot z_\delta(\tau),\dot z_\delta(\tau)\rangle 
+\calR^*_\delta(-\rmD\calE(\tau,z_\delta(\tau)))\d\tau =\calE(0,z_0). 
\end{align}
with $\calE(t,z):=\calI(z) -\langle\ell_*,z\rangle$. Finally, there exists a 
function 
$m:\calZ\times\calV^*\to[0,\infty)$ that maps bounded sets to bounded sets such 
that for all  $z_0\in \calZ$, $\ell_*\in\calV^*$ and $\delta>0$ the 
corresponding solution $z_\delta$ satisfies  
\begin{align}
\label{est.basicunifdelta1}
 \norm{z_\delta}_{L^\infty(\R_+;\calZ)}&\leq 
 m(z_0,\ell_*)\,,\\
\int_0^\infty \calR_\delta(\dot z_\delta(\tau)
+ \calR_\delta^*(-\rmD\calE(\tau,z_\delta(\tau))\d\tau&\leq  m(z_0,\ell_*)\,.
\label{est.basicunifdelta2}
\end{align}
\end{proposition}

\begin{proof} 
Existence of solutions: The arguments follow closely those presented in 
Section 4 of the preprint version \cite{KnRoZaPrep13} of 
\cite{KneesRossiZanini}. The operator 
$\calG_\delta:=(\partial\calR_\delta)^{-1}
:\calZ^*\rightrightarrows \calZ$ in fact is single valued and Lipschitz 
continuous. Thus, the differential inclusion 
\eqref{eq.viscdelta} is equivalent to the abstract ordinary differential 
equation $ \dot z_\delta(t) =\calG_\delta(-\rmD\calI(z_\delta(t)) +\ell_*)$ 
 living  in the space $\calZ$. Since $\rmD\calI:\calZ\to\calZ^*$ 
is locally Lipschitz continuous, by the Picard--Lindel\"of Theorem, for every 
initial value $z_0\in \calZ$ there exists a 
unique local solution $z_\delta\in 
W^{2,\infty}([0,T_0];\calZ)$. By standard convex analysis and chain rule 
arguments, see for instance \cite[Sec.\ 1.3.4]{MieRou15}, it follows that this 
solution satisfies the energy dissipation balance \eqref{eq.edb-delta} on 
$[0,T_0]$. 
In order to show that there is a global in time solution,  
assume that there exists $T_*>0$ such that the solution cannot be extended 
beyond $T_*$. By the energy-dissipation estimate it follows that 
$\norm{z_\delta}_{L^\infty((0,T_*);\calZ)}<\infty$ as well as $\hat z_\delta\in 
H^1((0,T_*);\calZ)$, implying in particular that $\hat 
z_\delta\in C([0,T_*];\calZ)$. Applying again the Picard-Lindel\"of theorem 
with the new initial value  $\hat z_\delta(T_*)$ we obtain a contradiction to 
the definition of $T_*$. 

The proof of \eqref{est.basicunifdelta1}--\eqref{est.basicunifdelta2} is an 
immediate consequence of the energy-dissipation balance, compare also the proof 
of Proposition \ref{prop.basicprop.paramsol}. 
\end{proof}

For solutions $z_\delta$ and $t>0$ we  define 
\[
\nu_\delta(t):=\left(\norm{\dot 
z_\delta(t)}_\bbV^2 +\delta\norm{ \dot z_\delta(t)}_A^2 \right)^\frac{1}{2},
\]
where $\norm{ v}_A:=\sqrt{\langle Av,v\rangle}$ is a norm on $\calZ$ that is 
equivalent to the standard norm on $\calZ$.  Since $\dot z_\delta\in 
W^{1,\infty}((0,T);\calZ)$, the function $\nu_\delta$ is well defined for all 
$t\in [0,\infty)$. 
\begin{proposition}
 \label{prop.regestregularizedRplus}
 Under the assumptions of Proposition \ref{prop.exregularizedRplus} for 
$\delta>0$, $\ell_*\in\calV^*$, $z_0\in \calZ$ with $\rmD\calI(z_0)\in 
\calV^*$  let $z_\delta\in W^{2,\infty}(\R_+;\calZ)$ with 
$z_\delta(0)=z_0$ be the unique solution of 
\eqref{eq.viscdelta}. Then 
\begin{align}
\label{est.initialdatumdelta}
 \nu_\delta(0)\leq \dist_\bbV(-\rmD\calI(z_0)+\ell_*,\partial\calR(0)).
 \end{align}
 Moreover, there exists a function $m:\calZ\times\calV^*\to[0,\infty)$ 
mapping bounded sets on bounded sets such that for all $\delta>0$, 
$\ell_*\in\calV^*$, $z_0\in \calZ$ with $\rmD\calI(z_0)\in \calV^*$ the 
corresponding solution satisfies
\begin{align}
  \norm{\dot z_\delta}_{L^\infty(\R_+;\calV)} + \sqrt{\delta}\norm{\dot 
z_\delta}_{L^\infty(\R_+;\calZ)} +\norm{\dot z_\delta}_{L^1(\R_+;\calZ)} 
&\leq m(z_0,\ell_*)(1 + \dist_\bbV(-\rmD\calI(z_0)+\ell_*,\partial\calR(0)))\,.
\label{est.initialdatumdeltb}
\end{align}
\end{proposition}

\begin{proof}
 Thanks to the regularity $z_\delta\in W^{2,\infty}(\R_+;\calZ)$ and the 
continuity of the quantities appearing in \eqref{eq.viscdelta}, relation  
\eqref{eq.viscdelta} in particular is valid for $t=0$. Let $\mu\in 
\partial\calR(0)$ such that 
$\dist_\bbV(-\rmD\calE(0,z_0),\partial\calR(0))= 
\norm{\rmD\calE(0,z_0)+\mu}_{\bbV^{-1}}$. Choosing $\dot z_\delta(0)$ as a test 
in \eqref{eq.viscdelta} for 
$t=0$ and exploiting the one-homogeneity of $\calR$ we find
\begin{align*}
 \calR(\dot z_\delta(0)) + \nu_\delta(0)^2 &=-\langle \rmD\calE(0,z_0),\dot 
z_\delta(0)\rangle 
\\
&=-\langle \rmD\calE(0,z_0)+\mu,\dot 
z_\delta(0)\rangle +\langle\mu,\dot z_\delta(0)\rangle
\\
&\leq \norm{\rmD\calE(0,z_0)+\mu}_{\bbV^{-1}}\norm{\dot 
z_\delta(0)}_{\bbV} 
+\calR(\dot z_\delta(0)),
\end{align*}
where in the last term we have exploited the one-homogeneity of 
$\calR$. Applying Young's inequality on the right hand side and 
absorbing corresponding terms yields \eqref{est.initialdatumdelta}.

Following the difference quotient arguments in \cite[Section 4]{KnRoZaPrep13}, 
the preprint version of \cite{KneesRossiZanini}, it follows that for almost all 
$t>0$ we have
\begin{align}
\label{est.prdelta001}
 \langle\bbV\ddot z_\delta(t),\dot z_\delta(t)\rangle + \delta\langle A\ddot 
z_\delta(t),\dot z_\delta(t)\rangle + \langle \tfrac{\rmd}{\rmd 
t}\rmD\calI(z_\delta(t)),\dot z_\delta(t)\rangle=0.
\end{align}
Observe that the first two terms on the left hand side coincide with 
$\frac{1}{2}\frac{\rmd}{\rmd t}\nu_\delta(t)^2$. 
With the same interpolation argument as in the proof of 
Lemma \ref{lem.estDF}, see also \cite[Lemme 1.1]{Knees2018a}, 
we obtain 
\begin{align*}
 \abs{\rmD^2\calF(z_\delta(t))[\dot z_\delta(t),\dot z_\delta(t)]}
 &\leq c(1+\norm{z_\delta(t)}_\calZ^q)\norm{\dot z_\delta(t)}_{\calV} 
\norm{\dot z_\delta(t)}_{\calZ}\\
&\leq \tfrac{\alpha}{2}\norm{\dot z_\delta(t)}_\calZ^2 + 
c_\alpha c_{\norm{z_\delta}_{L^\infty(\R_+;\calZ)}}\calR(\dot z_\delta(t))
\norm{\dot z_\delta(t)}_\calV.
\end{align*}
Taking into account the uniform bound \eqref{est.basicunifdelta1} and going 
back to \eqref{est.prdelta001} we have shown that there exists a function 
$\wt m:\calZ\times\calV^*\to [0,\infty)$ mapping bounded sets on bounded sets 
such that for all $\delta>0$, $z_0\in \calZ$, $\ell_*\in \calV^*$ we have 
\begin{align}
\label{est.prdelta002}
 \frac{1}{2}\frac{\rmd}{\rmd t}\nu_\delta(t)^2 +\tfrac{\alpha}{2}\norm{\dot 
z_\delta(t)}_\calZ^2\leq \wt m(z_0,\ell_*)\calR(\dot z_\delta(t))\norm{\dot 
z_\delta(t)}_\calV.
\end{align}
Clearly, $\norm{\dot z_\delta(t)}_\calV\leq \nu_\delta(t)$. Moreover, since 
$\calZ$ is continuously embedded in $\calV$, for $\delta\leq 1$ 
we have $\norm{\dot z_\delta(t)}_\calZ\geq c(\norm{\dot z_\delta(t)}_\calZ^2 
+\norm{\dot z_\delta(t)}_\calV^2)^\frac{1}{2}\geq c\nu_\delta(t)$, and $c$ is 
independent of $\delta>0$ and $\dot z_\delta$. Hence, \eqref{est.prdelta002}  
can be rewritten as follows 
\begin{align}
 \frac{1}{2}\frac{\rmd}{\rmd t}\nu_\delta(t)^2 + c\norm{\dot 
z_\delta(t)}_\calZ\nu_\delta(t)\leq\wt m(z_0,\ell_*)\calR(\dot 
z_\delta(t))\nu_\delta(t).
\end{align}
Next we follow the arguments from \cite[Section 4.4]{Mie11}. For $t>0$ 
with $\nu_\delta(t)\neq 0$ we find
\begin{align*}
 \dot\nu_\delta(t) + c\norm{\dot z_\delta(t)}_\calZ
\leq \wt m(z_0,\ell_*)\calR(\dot z_\delta(t)),
 \end{align*}
 while for $t$ with $\nu_\delta(t)=0$ the previous inequality is trivially 
satisfied. Integration with respect to $t$  yields
\begin{align*}
\forall t>0\quad 
\nu_\delta(t) +\int_0^t\norm{\dot z_\delta(\tau)}_\calZ\d\tau\leq 
c\Big(\nu_\delta(0) +\wt m(z_0,\ell_*)\int_0^\infty \calR(\dot 
z_\delta(\tau))\d\tau\Big)\,.
\end{align*}
Combining this with estimate \eqref{est.initialdatumdelta} and 
\eqref{est.basicunifdelta2} we finally have shown that there exists a further 
function $m:\calZ\times \calV^*\to[0,\infty)$ mapping bounded sets on bounded 
sets such that \eqref{est.initialdatumdeltb} is valid.
\end{proof}

\begin{proof}[Proof of Theorem \ref{thm.viscRplusexregest}]
 \textit{Uniqueness of solutions:} 
For $i\in \{1,2\}$ let $z_i\in L^\infty(\R_+;\calZ)$ with $\dot z_i\in 
L^1((0,a);\calV)$ for all $a>0$ and such that \eqref{eq.viscsysRplus} is 
satisfied. Testing \eqref{eq.viscsysRplus} with the difference $z_1(t) - 
z_2(t)$, by the monotonicity of 
the operator $\partial\calR$ we obtain
\begin{align*}
 0\geq \langle \bbV(\dot z_1(t) - \dot z_2(t)), z_1(t) - z_2(t)\rangle + 
 \langle \rmD\calI(z_1(t)) -\rmD\calI(z_2(t)),z_1(t) - z_2(t)\rangle.
\end{align*}
Thanks to the $\lambda$-convexity estimate  \eqref{est.lambda-convex} this 
implies
\begin{align*}
 0\geq \tfrac{1}{2}\tfrac{\rmd}{\rmd t}\norm{z_1(t) - 
z_2(t)}_\bbV^2 + \tfrac{\alpha}{2}\norm{z_1(t) - z_2(t)}^2_\calZ - 
\lambda\norm{z_1(t) - z_2(t)}_\calV^2
\end{align*}
for some $\lambda>0$. 
Integration with respect to $t$ and applying the Gronwall estimate we conclude.

 \textit{Existence and regularity of solutions:} Let $z_0\in \calZ$ with 
$\rmD\calI(z_0)\in \calV^*$ and $\ell_*\in \calV^*$. 
 For $\delta>0$ let $z_\delta$ denote the corresponding solution of 
\eqref{eq.viscdelta}. Thanks to the uniform estimates provided in Proposition 
\ref{prop.exregularizedRplus} and Proposition 
\ref{prop.regestregularizedRplus}, there exists a vanishing sequence 
$(\delta_n)_n$ (i.e.\ $\delta_n\to 0$ for $n\to\infty$) and a function $z\in 
L^\infty(\R_+;\calZ)$ with $\dot z\in L^\infty(\R_+;\calV)\cap 
L^2(\R_+;\calV)$ and $\Var_\calZ(z,[0,\infty))<\infty$ such that the following 
convergences are available:
\begin{gather}
 z_{\delta_n}\overset{*}{\rightharpoonup} z \text{ weakly$*$ in 
}L^\infty(\R_+;\calZ),\\
\dot z_{\delta_n}\overset{*}{\rightharpoonup} \dot z \text{ weakly$*$ in 
}L^\infty(\R_+;\calV)\cap L^2(\R_+;\calV),\\
\liminf_{\delta_n}\norm{\dot 
z_{\delta_n}}_{L^1(\R_+;\calZ)}=\liminf_{\delta_n}\Var_\calZ(z_{\delta_n},[0,
\infty))\geq \Var_\calZ(z,[0,\infty)),\\
\text{for all 
$t\geq 0$:}\qquad  
z_{\delta_n}(t)\rightharpoonup z(t) \text{ weakly in }\calZ 
\,.
\label{eq.convdelta4}
\end{gather}
 The last two assertions are a consequence of the Banach space valued version 
of Helly's selection principle, \cite{BarbuPrecupanu86}. Moreover, by lower 
semicontinuity, the uniform bounds derived in Proposition 
\ref{prop.exregularizedRplus} and Proposition 
\ref{prop.regestregularizedRplus} carry over to the limit function $z$, whence 
\eqref{est.viscRp1}--\eqref{est.viscRp2}. In the following we omit the index 
$n$.

Let us next prove that for almost all $t>0$ the function  $z$ satisfies the 
inclusion 
\eqref{eq.viscsysRplus}. For that purpose we start from the energy dissipation 
balance \eqref{eq.edb-delta}. 
Thanks to \eqref{eq.convdelta4} and since $\rmD\calI:\calZ\to\calZ^*$ is weakly 
continuous,   pointwise 
weak convergence in $\calZ^*$ of the terms $(\rmD\calI(z_{\delta_n}(\cdot)))_n$ 
 ensues. Hence, by \cite[Lemma A.1]{KnRoZaPrep13}, see also Lemma 
\ref{app_prop:lsc3} in 
the appendix, we obtain for all $t>0$
\begin{align*}
 \liminf_{\delta}\calR_\delta^*(-\rmD\calI(z_\delta(t))+\ell_*)\geq 
 \calR_\bbV^*(-\rmD\calI(z(t))+\ell_*).
\end{align*}
Hence, by lower semicontinuity and Fatou's Lemma it follows that the limit 
function $z$ satisfies the energy-dissipation estimate
\begin{align}
\label{eq.ede-Rplus}
\forall t\geq 0\qquad  \calE(t,z(t)) + \int_0^t\calR_\bbV(\dot z(\tau)) 
+\calR_\bbV^*(-\rmD\calE(\tau,z(\tau))\d\tau\leq \calE(0,z_0). 
\end{align}
Standard arguments relying on the chain rule (Proposition 
\ref{prop.chainrules}) show that we in fact have an equality in 
\eqref{eq.ede-Rplus}. Applying again the chain rule and localizing this 
energy-dissipation identity  shows that $z$ satisfies the inclusion 
\eqref{eq.viscsysRplus}. Since $\partial\calR(0)$ is a bounded subset of 
$\calV^*$, together with the estimate \eqref{est.viscRp2} we finally conclude 
that $\rmD\calI(z(\cdot))$ belongs to $L^\infty(\R_+;\calV^*)$ and satisfies 
\eqref{est.viscRp3}. This finishes the proof. 
\end{proof}

\subsubsection{A uniform estimate for the driving force $\rmD\hat\calE$}

We now turn back to the properties of the parametrized BV solutions 
introduced in Definition \ref{def:paramsol}.
\begin{theorem}
\label{thm.regforce}
  There exists a function $m:\calZ\times H^1((0,T);\calV^*)\to[0,\infty)$ 
mapping bounded sets to bounded sets such that  for all $z_0\in \calZ$, 
$\ell\in H^1(0,T;\calV^*)$ satisfying \eqref{eq.compatibledata} and all 
$(S,\hat t,\hat z)\in \calL(z_0,\ell)$ we have 
$\rmD\hat \calE(\cdot,\hat z(\cdot))\in L^\infty(0,S;\calV^*)$ and 
\begin{align*}
 \bnorm{\rmD\hat \calE(\cdot,\hat z(\cdot))}_{L^\infty(0,S;\calV^*)}
 + \norm{\lambda \bbV \hat z'}_{L^\infty(G;\calV^*)}
 &\leq m(z_0,\ell) 
\end{align*}
and $\rmD\calI(\hat z(\cdot))\in C_\text{weak}([0,S];\calV^*)$. 
\end{theorem}
\begin{remark}
As a byproduct, in the proof of Theorem \ref{thm.regforce} we show that the 
function  $\lambda$ from \eqref{eq.diffincllambda} 
is positive almost everywhere  on $G$, that  the function $s\mapsto 
1/\lambda(s)$ belongs to $L^1_\text{loc}(G)$ but that it  is not integrable on 
any connected component of $G$, see also Remark \ref{rem:heteroclorbits} 
for further consequences of this observation.
\end{remark}

\begin{proof}
As already stated at the beginning of this section, we have 
$\rmD\hat \calE(\cdot,\hat z(\cdot))\in L^\infty((0,S)\backslash G;\calV^*)$ 
with $\bnorm{\rmD\hat \calE(\cdot,\hat z(\cdot))}_{ L^\infty((0,S)\backslash 
G;\calV^*)}\leq \diam_{\calV^*}(\partial\calR(0))$ and it remains to study the 
behavior on the set $G$. 
For that purpose we start from the differential inclusion 
\eqref{eq.diffincllambda}. 
We recall that by the definition of 
$\vvp$-parametrized solutions the set $G$ is a relatively open subset of 
$[0,S]$. 
Let $(a,b)\subset G$ be a maximal connected component of 
$G$. By Proposition \ref{prop.basicprop.paramsol}, $\hat t$ is 
constant on $(a,b)$. Hence, $\hat\ell$ is constant on $(a,b)$ as well and we 
denote its value with $\ell_*$. 
For each compact set 
$K\subset(a,b)$ we have $\hat z\in W^{1,1}(K;\calV)$ and 
$\rmD\hat\calE(\cdot,\hat z(\cdot))\in L^\infty(K;\calV^*)$ which implies that 
$\rmD\hat \calE(\cdot,\hat z(\cdot))\in C_\text{weak}(K;\calV^*)$. Thus, by 
 lower semicontinuity, there exists $c_K>0$ such that for $\vvm(\cdot,\cdot)$ 
from \eqref{def.vvm} it holds  
$\vvm(\hat\ell(\cdot),\hat z(\cdot))\geq c_K>0$ for all $s\in K$.  
The normalization condition \eqref{eq.normalized.sol} now implies that 
$\norm{\hat z'(s)}_\calV\leq c_K^{-1}$ almost everywhere on $K$ and hence 
$\lambda(s)\geq c_K^2>0$ almost everywhere on $K$, where we used the 
representation of $\lambda$ from Proposition \ref{prop.basicprop.paramsol}.  
 This observation was already made in \cite{MRS16}.

The next aim is to perform a change of variables $s\mapsto r$ and  $(a,b)\to 
(0,\Lambda)$ such that  \eqref{eq.diffincllambda} rewritten in the new variable 
is of the form \eqref{eq.viscsysRplus}.  

Assume first that there is $s_*\in (a,b)$ such that $1/\lambda\notin 
L^1((a,s_*))$. 
The above considerations imply that for every $\varepsilon>0$ there exists a 
constant $c_\varepsilon>0$ such that 
$\lambda^{-1}\big|_{(a+\varepsilon,s_*)}\leq c_\varepsilon$. Hence, since 
$\lambda^{-1}$ is not integrable on $(a,s_*)$, $\lambda^{-1}$ is unbounded in 
a neighborhood of $a$. To be more precise, for every $n\in \N$ the set 
$\Sigma_n:=\Set{s\in (a,a+\frac{1}{n})}{1/\lambda(s)\geq n}$ has positive 
Lebesgue measure. From the normalization property and the structure of 
$\lambda$ we therefore deduce that
\begin{align}
 \text{for all $n\in \N$ and almost all $s\in \Sigma_n$:}\qquad 
 \dist_\bbV(-\rmD\hat\calE(s,\hat z(s)),\partial\calR(0))\leq 
\tfrac{1}{\sqrt{n}}.
\end{align}
Let now $s_n\in \Sigma_n$ with $\lambda(s_n)^{-1}\geq n$ and such that 
$\dist_\bbV(-\rmD\hat\calE(s_n,\hat z(s_n)),\partial\calR(0))\leq 
\tfrac{1}{\sqrt{n}}$.
 Clearly, $\lim_n s_n=a$ and without loss of generality we 
 may assume that  the sequence $(s_n)_n$ is decreasing. We next study the 
 system \eqref{eq.diffincllambda} on the intervals $(s_n,b)$. 
  For $s\in (s_n,b)$ let $\Lambda_n(s):=\int_{s_n}^s 
\frac{1}{\lambda(\sigma)}\d\sigma$. 
The above considerations show that 
 $\Lambda_n$ 
 is well defined on $(s_n,b)$. Moreover, $\Lambda_n$ is strictly 
 increasing and the inverse function 
$\Lambda_n^{-1}:[0,\Lambda_n(b))\to[s_n,b)$  exists. 
We remark that $\Lambda_n(b)=\infty$ is not excluded.
 For $r\in [0,\Lambda_n(b))$ let $\wt z_n(r):=\hat z(\Lambda_n^{-1}(r))$. 
Observe that $\wt z_n\in W^{1,1}(0,\Lambda_n(b-\delta);\calV^*)$ for every 
$\delta>0$.  The function  $\wt z_n$ solves the Cauchy problem
 \begin{align}
 &0\in \partial\calR(\wt z_n'(r)) +\bbV\wt z_n'(r) +\rmD\calI(\wt 
z_n(r))-\ell_*,  \quad r\in (0,\Lambda(b)),\\
&\wt z_n(0)=\hat z(s_n).
 \end{align}
Hence, Theorem \ref{thm.viscRplusexregest} is applicable and implies in 
particular that $\rmD\calI(\wt z_n(\cdot))\in L^\infty(0,\Lambda_n(b);\calV^*)$ 
with 
\begin{align*}
 \norm{\rmD\calI(\wt z_n(\cdot))}_{L^\infty(0,\Lambda_n(b);\calV^*)}
 \leq m_2(\hat z(s_n),\ell_*)\big(\dist_\bbV(-\rmD\hat\calE(s_n,\hat 
z(s_n)),\partial\calR(0))+m_1(\hat z(s_n),\ell_*) ),
\end{align*}
where $m_1,m_2:\calZ\times \calV^*\to[0,\infty)$ are functions that map bounded 
sets on bounded sets and that do not depend on $n$. This immediately translates 
into $\rmD\calI(\hat z(\cdot))\in L^\infty(s_n,b;\calV^*)$ along with the 
estimate 
\begin{align}
\label{est.proofmthm1001}
 \norm{\rmD\calI(\hat z(\cdot))}_{L^\infty(s_n,b;\calV^*)}
 &\leq m_2(\hat z(s_n),\ell_*)\big(\dist_\bbV(-\rmD\hat\calE(s_n,\hat 
z(s_n)),\partial\calR(0))+m_1(\hat z(s_n),\ell_*) \big)\\
&\leq \wt m_2(z_0,\ell)
\big(\tfrac{1}{\sqrt{n}}+\wt 
m_1(z_0,\ell) \big),
\end{align}
where $\wt m_1,\wt m_2:\calZ\times H^1((0,T);\calV^*)\to[0,\infty)$ are 
functions that map bounded sets on bounded sets and depend on $\calI,\calR$ 
and embedding constants, only.  
The previous estimate is of the structure $\alpha_n\leq \beta_n$ with an 
increasing sequence $(\alpha_n)_n$ and a decreasing sequence $(\beta_n)_n$. 
Hence,  for 
 $s_n\searrow 0$ we  obtain $\rmD\calI(\hat z(\cdot))\in 
L^\infty(a,b;\calV^*)$ along with a bound that ultimately depends on 
$\norm{z_0}_\calZ$ and $\norm{\ell}_{H^1(0,T;\calV^*)}$, only. 

Assume next that $\lambda^{-1}\in L^1(a,s_*)$ for every $s_*<b$. In this case, 
we use the transformation 
$\Lambda(s):=\int_a^s 
\frac{1}{\lambda(\sigma)}\d\sigma$ and the transformed function $\wt z$ 
satisfies \eqref{eq.viscsysRplus} on $(0,\Lambda(b))$ with the initial 
condition $\wt z(0)=\hat z(a)$. Since $G$ is open, $a$ does not belong to $G$ 
and hence, $-\rmD\hat\calE(a,\hat z(a))\in \partial\calR(0)$. According to 
Remark \ref{rem.constsol} the unique solution of the transformed system is 
given by the constant function $\wt z(r)=\hat z(a)$ for all $r\in 
(0,\Lambda(b))$. But this implies in particular that $\hat z$ is constant on 
$(a,b)$, a contradiction to the normalization condition. As a consequence, 
 $\lambda^{-1}$ is not bounded close to $a$. Similarly, 
again taking into account Remark \ref{rem.constsol} it follows that 
the values $\Lambda_n(b)$ from above and the value $\Lambda(b)$ in the 
situation discussed here, are not finite, since otherwise $-\rmD\calI(\wt 
z_n(\Lambda_n(b)))+\ell_*\in \partial\calR(0)$. Summarizing this shows that 
 the function $s\mapsto 1/\lambda(s)$ is not integrable on $(a,b)$  and 
unbounded towards $a$ and $b$.

If $0\notin G$, then the proof of Theorem \ref{thm.regforce} is finished. 
Otherwise let $[0,b)$ be a maximal connected component of $G$.   But now we can 
argue exactly in the same way as before with $0$ instead of $s_n$ in 
\eqref{est.proofmthm1001}. 
\end{proof}

\begin{remark}
 \label{rem:heteroclorbits} 
 Reinterpreting the arguments of the previous proof we have shown the 
following: Let $(S,\hat t,\hat z)\in \calL(z_0,\ell)$ and let $(a,b)\subset G$ 
be a maximal connected component of $G$. Let $s_*:=(a+b)/2$ and define 
$\Lambda_*(s):=\int_{s_*}^s \frac{1}{\lambda(r)}\d r$. The arguments from the 
previous proof show that $\Lambda_*$ is well defined, strictly monotone and that 
 $\Lambda_*((a,b))=\R$ with $\lim_{s\to a}\Lambda_*(s) =-\infty$ and 
$\lim_{s\to b}\Lambda_*(s)=\infty$.  
 Moreover, $\wt z:=\hat z\circ\Lambda_*^{-1}$ (inverse function) satisfies 
\eqref{eq.viscsysRplus} on $\R$ with $\lim_{r\to -\infty}\wt z(r)=\hat z(a)$ and 
$\lim_{r\to \infty}\wt z(r)=\hat z(b)$ (strong convergence in $\calV$ since 
$\hat z\in C([0,S];\calV)$). The limit points $\hat z(a),\hat z(b)$ are stable 
in the sense that $-\rmD\calI(z_*) +\ell_*\in \partial\calR(0)$ for $z_*\in 
\{\hat z(a),\hat z(b)\}$. Hence, $\wt z$ can be interpreted as a heteroclinic 
orbit for \eqref{eq.viscsysRplus}, connecting $\hat z(a)$ and $\hat z(b)$.
\end{remark}


\subsection{Compactness of solution sets}\label{subsec:compactness}
The aim of this section is to derive compactness properties of the sets 
\begin{align}
\label{eq.defMrho}
M_\varrho:=\Set{(S,\hat t,\hat z)}{(S,\hat t,\hat z)\in 
\calL(z_0,\ell) \text{ for $(z_0,\ell)$ with 
\eqref{eq.compatibledata} and }\norm{z_0}_\calZ + 
\norm{\ell}_{H^1((0,T);\calV^*)}\leq \varrho}.
\end{align} for arbitrary $\rho\geq 0$. These properties will be based 
on the 
uniform estimates derived in the previous two sections.
\begin{theorem}\label{thm.properties_sol_set}
Let $\varrho>0$ and $z_0\in \calZ$. 
Then the set $M_\varrho$ is compact in the following sense: For every sequence 
$(S_n,\hat t_n,\hat z_n)_{n\in\N}\subseteq M_\varrho$ with $(S_n,\hat t_n,\hat 
z_n)\in \calL(z_0,\ell_n)$ and such that $(z_0,\ell_n)$ satisfy 
\eqref{eq.compatibledata}, there exists a subsequence  
(denoted by the same symbols for simplicity) and limit elements $\ell\in 
H^1(0,T;\calV^*)$ and $(S,\hat t,\hat z)\in\calL(z_0,\ell)$ such that 
$(z_0,\ell)$ comply with \eqref{eq.compatibledata} and 
\begin{align}
S_n\to S \text{ in } \bbR,\quad \hat t_n\overset{*}{\rightharpoonup}\hat t 
\text{ in } W^{1,\infty}(0,S),\quad \hat t(S)=T,\quad \ell_n\rightharpoonup\ell 
\text{ in }H^1(0,T;\calV^*),\label{est.convparam1}\\
\hat z_n\overset{*}{\rightharpoonup}\hat z \text{ in } L^\infty(0,S;\calZ) 
\text{ and }\hat z_n\to\hat z \text{ uniformly in } 
C([0,S],\calV),\label{est.convparam2}\\
\hat z_n(S_n)\to\hat z(S) \text{ strongly in }\calV, 
\label{est.convparam5}\\
\rmD\calI(\hat z_n)\overset{*}{\rightharpoonup} \rmD\calI(\hat z) \text{ in } 
L^\infty(0,S;\calV^\ast),\label{est.convparam3}
\end{align}
and for every $s\in[0,S]$, it holds that
\begin{align}
\hat t_n(s)\to\hat t(s),\quad \hat z_n(s)\rightharpoonup \hat z(s) \text{ in } 
\calZ,\quad \rmD\calI(\hat z_n(s))\rightharpoonup \rmD\calI(\hat z(s))\text{ in 
}\calV^\ast,\label{est.convparam4}
\\
\hat z_n(s)\rightarrow\hat z(s)\,\,\text{ strongly in }\calZ. 
\label{est.convparam6}
\end{align}
Furthermore, the map $s\mapsto\rmD\calI(\hat z(s))$ is continuous w.r.t. the 
weak topology on $\calV^*$.
\end{theorem}

\begin{proof}
Let $(S_n,\hat t_n,\hat z_n)_{n\in\N}\subseteq M_\varrho$ be a sequence as in 
the proposition and for 
$n\in\N$ let $G_n\subset[0,S]$ be the corresponding open sets according to 
Definition \ref{def:paramsol}.  Thanks to Proposition 
\ref{prop.basicprop.paramsol}, the estimates \eqref{eq.unif-est1-paramsola} and 
\eqref{eq.unif-est1-paramsolb} hold uniformly for $n\in\N$ an we infer the first 
of \eqref{est.convparam1}. If $S>S_n$, we extend all functions $\hat z_n$ and 
$\hat t_n$ constantly to $[0,S]$ by their value at $S_n$ and thus obtain the 
first of \eqref{est.convparam2}. Due to \eqref{eq.pparamsol-con1} and the 
normalization condition \eqref{eq.normalized.sol}, the second of 
\eqref{est.convparam1} ensues, and since $W^{1,\infty}(0,S)$ is compactly 
embedded into $C([0,S])$, also the third of \eqref{est.convparam1} as well as 
the first of \eqref{est.convparam4}. Combining the a priori estimate 
\eqref{eq.unif-est1-paramsola} and the normalization condition 
\eqref{eq.normalized.sol}, we conclude uniform convergence of $\hat z_n$ to 
$\hat z$ in $\calV$ and pointwise weak convergence in $\calZ$ along a 
subsequence by means of Proposition  
\ref{prop.hellyarzasc-version2}.  We also obtain \eqref{est.convparam5} by 
means of the following estimate:
\begin{align*}
\Vert \hat z_n(S_n)-\hat z(S)\Vert_\calV\leq \Vert \hat z_n(S_n)-\hat 
z_n(S)\Vert_\calV+\Vert\hat z_n(S)-\hat z(S)\Vert_\calV\to 0, 
\end{align*}
where for the convergence of the first term we exploit the equicontinuity 
of the sequence $(\hat z_n)_n$ (cf.\ the proof of Proposition 
\ref{prop.hellyarzasc-version2}) 
and the second 
summand tends to zero due to the uniform convergence \eqref{est.convparam2}. 

In order to show 
\eqref{est.convparam3}, we 
first note that thanks to the a priori estimate in Theorem \ref{thm.regforce}, there are an element 
$\xi\in\calV^\ast$ such that $\rmD \calI(\hat 
z_n)\overset{*}{\rightharpoonup}\xi$ in $L^\infty(0,S;\calV^\ast)$ as well as 
pointwise limits such that $\rmD \calI(\hat z_n(s)){\rightharpoonup}\mu(s)$ in 
$\calV^*$ for all $s\in[0,S]$ along a subsubsequence. Now, since we also have 
$\hat z_n(s)\rightharpoonup \hat{z}(s)$ in $\calZ$, and  
$\rmD\calF$ is supposed to be weakly continuous (cf. \eqref{ass.fweakconv}), we 
also know that $\rmD \calI(\hat z_n(s)){\rightharpoonup}\rmD\calI(\hat z(s))$ 
in 
$\calZ^\ast$, whence \eqref{est.convparam3} and the third of 
\eqref{est.convparam4} ensue along a subsequence. A standard argument  by 
contradiction shows convergence along the entire sequence. By the same 
arguments, we obtain the weak 
continuity of $s\mapsto \rmD\calI(\hat z(s))$.

It remains to show that $(S,\hat t,\hat z)\in \calL(z_0,\ell)$. As a first 
step, 
we show that the complementarity identity \eqref{eq.complementarity} is valid. 
To this end, we note that $\hat{t}_n^\prime\rightharpoonup\hat t^\prime$ in 
$L^1(0,S)$. Furthermore, we have $\ell_n(\hat t_n(s))\rightharpoonup \ell(\hat t(s))$ 
in 
$\calV^*$ for all $s\in[0,S]$ according to Lemma \ref{lemma:convell}. Together with the weak 
convergence of $\rmD\calI(s,\hat z_n(s))$ according to \eqref{est.convparam4} 
and the weak lower semicontinuity of 
$\mathrm{dist}_{\bbV}(\cdot,\partial\calR(0))$, this implies
\begin{align}
\vvm(\ell(\hat t(s)),\hat z(s))\leq \liminf_{n\to\infty}\vvm(\ell_n(\hat 
t_n(s)),\hat z_n(s)) \text{ for all }s\in[0,S]\label{convparam.m}
\end{align}
with $\vvm(\cdot,\cdot)$ from \eqref{def.vvm}. This  
 allows us to conclude by means of Lemma \ref{app_prop:lsc2} that we have 
\begin{align*}
0\leq\int_0^S \hat t^\prime(s)  \vvm(\ell(\hat t(s)),\hat z(s))\ds\leq 
\liminf_{n\to\infty}\int_0^S\hat t^\prime_n(s)\vvm(\ell_n(\hat t_n(s)),\hat 
z_n(s))\ds=0,
\end{align*}
and since the integrand is nonnegative, \eqref{eq.complementarity} ensues.

Next, we want to show that \eqref{eq.energydissipidentitylimitpparam} is valid 
with $\leq$ instead of $=$. For every $n\in\N$ and $s\in[0,S]$, it holds with 
$\hat\calE_n(s,v):=\calI(v)-\langle \ell_n(s),v\rangle$ and 
$\hat\ell_n:=\ell_n\circ\hat t_n$ that
\begin{align*}
\hat\calE_n(s,\hat z_n(s))+\int_0^s\calR[\hat z_n^\prime](r)\dr+\int_{[0,s]\cap 
G_n}\Vert \hat z_n^\prime(r)\Vert_{\bbV}&\vvm(\hat\ell_n(r),\hat z_n(r))\dr\\
&=\hat\calE_n(0,z_0)-\int_0^s\langle\hat\ell^\prime_n(r),\hat z_n(r)\rangle\dr.
\end{align*}
Now, the second of \eqref{est.convparam4} together with the lower semicontinuity 
of $v\mapsto \calI(v)$ w.r.t. the weak topology on  
$\calZ$, as well as Lemma \ref{lemma:convell} imply for all $s\in[0,S]$ that 
\begin{align*}
\liminf_{n\in\N}\hat\calE_n(s,\hat z_n(s))\geq \hat\calE(s,\hat z(s)) \text{ 
and 
} \lim_{n\to\infty}\hat\calE_n(0, z_0)=\hat\calE(0,z_0).
\end{align*}
For the first dissipation integral, it follows by means of Helly's selection 
principle, \cite[Theorem 3.2]{MaMi05}, for all $s\in[0,S]$ that
\begin{align}
\liminf_{n\to\infty}\int_0^s\calR[\hat z_n^\prime](r)\dr\geq \int_0^s\calR[\hat 
z^\prime](r)\dr.\label{convparam.dissip1}
\end{align}
According to Lemma \ref{lemma:convell}, the load term fulfills the convergence
\begin{align*}
\int_0^s\langle\hat\ell^\prime(r),\hat 
z(r)\rangle\dr=\lim_{n\to\infty}\int_0^s\langle\hat\ell^\prime_n(r),\hat 
z_n(r)\rangle\dr,
\end{align*}
and it remains to study the second dissipation term. Let $G:=\lbrace 
s\in[0,S]:\, \vvm(\hat \ell(s),\hat z(s))>0\rbrace$. First, we show that $G$ is 
an open set. To this end, let $(s_k)_{k\in\N}\subset[0,S]\setminus G$ be a 
sequence converging to an element $s\in[0,S]$. By the weak continuity of 
$s\mapsto \rmD\calI(\hat z(s))$, we obtain 
$0=\liminf_{n\to\infty}\vvm(\hat\ell(s_n),\hat z(s_n))\geq 
\vvm(\hat\ell(s),\hat 
z(s))=0$, whence $s\in[0,S]\setminus G$ and $G$ is indeed open. Next, we are 
going to show the improved regularity of $\hat z$ on $G$. Let $K\subset G$ be 
compact. By the same arguments as above, we conclude that $c:=\liminf_{s\in 
K}\vvm( \hat\ell(s),\hat z(s))>0$. 
Thus, for every $s\in K$, there exists $N_0\in\N$ such that for all $n\geq N_0$ 
we have $\vvm(\ell_n(\hat t_n(s)),\hat z_n(s))\geq \frac c2$, and a proof by 
contradiction shows that $N_0$ can be chosen independently of $s\in K$. 
Therefore,  the normalization condition \eqref{eq.normalized.sol} shows that 
$\sup_{n\geq N_0}\Vert \hat z^\prime_n\Vert_{L^\infty(K;\bbV)}\leq\frac 2c$, 
whence it follows in combination with \eqref{est.convparam2} that $\hat 
z_n\overset{*}\rightharpoonup\hat z\in W^{1,\infty}(K;\calV)$.  
Now, 
by means of Proposition \ref{app_prop:lsc}  and having in mind 
\eqref{convparam.m}, we may conclude that
\begin{align}
\liminf_{n\to\infty}\int_{K}\Vert \hat 
z_n^\prime(r)\Vert_{\bbV}\vvm(\hat\ell_n(r),\hat z_n(r))\dr\geq \int_{K}\Vert \hat 
z^\prime(r)\Vert_{\bbV}\vvm(\hat\ell(r),\hat z(r))\dr,\label{convparam.dissip2}
\end{align}
and thus, \eqref{eq.energydissipidentitylimitpparam} is valid with $\leq$ 
instead of $=$ and also \eqref{eq.normalized.sol} with $\geq$ instead of $=$. 

In order to show the opposite estimates, we follow the ideas from 
\cite{KnZa18prep}. 

We first show that $s\mapsto \calI(\hat z(s))$ is continuous on $[0,S]$ 
and hence uniformly continuous. From $\hat z\in C([0,S];\calV)\cap 
L^\infty(0,S;\calZ)$ we obtain $\hat z\in C_\text{weak}([0,S];\calZ)$. 
Hence, thanks to \eqref{ass.fweakconv}, $\calF(\hat z(\cdot))$ is continuous on 
$[0,S]$ and $\rmD\calF(\hat z(\cdot))$ belongs to 
$C_\text{weak}([0,S];\calV^*)$. Since the same is true for $\rmD\calI(\hat 
z(\cdot))$, we conclude that  $A\hat z(\cdot)$ is continuous  with respect 
to the weak topology in $\calV^*$, as well.  But this ensures the continuity of 
the term $s\mapsto\langle A\hat z(s),\hat z(s)\rangle_{\calV^*,\calV}$ and 
ultimately the continuity of $\calI(\hat z(\cdot))$.

For 
$s\in[0,S]$, let $\mu(s)\in\partial\calR(0)$ such that $\Vert-\rmD_z\hat \calE 
(s,\hat z(s))-\mu(s)\Vert_{\bbV^{-1}}=\vvm(\hat \ell(s),\hat z(s))$. An 
application of \eqref{est.lambda-convex-energy} yields for every $s\in[0,S)$ and 
$0<h<S-s$
\begin{align*}
\calI(\hat z(s+h))-\calI(\hat z(s))\geq& \langle \rmD_z\hat
\calE(s,\hat z(s)),\Delta_h\hat z(s)\rangle+\langle 
\hat\ell(s),\Delta_h\hat 
z(s)\rangle-\lambda\calR(\Delta_h\hat z(s))\Vert\Delta_h\hat z(s)\Vert_\bbV\\
=&\langle \rmD_z\hat\calE(s,\hat z(s))+\mu(s),\Delta_h\hat z(s)\rangle+\langle 
\hat\ell(s),\Delta_h\hat z(s)\rangle-\langle\mu(s),\Delta_h\hat z(s)\rangle\\
&-\lambda\calR(\Delta_h\hat z(s))\Vert\Delta_h\hat z(s)\Vert_\bbV,
\end{align*}
where we abbreviate $\Delta_h\hat z(s):=\hat z(s+h)-\hat z(s)$. Now, thanks to 
the choice of $\mu(s)$, we can estimate the first term on the  
right hand side by
\begin{align*}
-\langle \rmD_z\hat\calE(s,\hat z(s))+\mu(s),\Delta_h\hat z(s)\rangle&\leq\Vert 
\rmD_z\hat\calE(s,\hat z(s))+\mu(s)\Vert_{\bbV^{-1}}\Vert\Delta_h\hat 
z(s)\Vert_\bbV=\vvm(\hat\ell(s),\hat z(s))\Vert\Delta_h\hat z(s)\Vert_\bbV,
\end{align*}
and the third term by $\langle\mu(s),\Delta_h\hat 
z(s)\rangle\leq\calR(\Delta_h\hat z(s))$. Therefore, rearrangement of the terms 
leads to the estimate
\begin{align*}
\calI(\hat z(s+h))-\calI(\hat z(s))+\vvm(\hat\ell(s),\hat z(s))\Vert\Delta_h\hat 
z(s)\Vert_\bbV+(1+\lambda\Vert\Delta_h\hat z(s)\Vert_\bbV)\calR(\Delta_h\hat 
z(s))\geq \langle \hat\ell(s),\Delta_h\hat z(s)\rangle,
\end{align*}
which we divide by $h>0$ and integrate w.r.t. $s$ to obtain for every 
$0\leq\sigma_1<\sigma_2\leq S-h$
\begin{align}
&\int_{\sigma_1}^{\sigma_2}\tfrac1h\bigl(\calI(\hat z(s+h))-\calI(\hat 
z(s))\bigr)\ds\nonumber\\
&+\int_{\sigma_1}^{\sigma_2}\vvm(\hat\ell(s),\hat z(s))\Vert\tfrac1h\Delta_h\hat 
z(s)\Vert_\bbV\ds
+\int_{\sigma_1}^{\sigma_2}(1+\lambda\Vert\Delta_h\hat 
z(s)\Vert_\bbV)\calR(\tfrac1h\Delta_h\hat z(s))\ds\nonumber\\
&\geq\int_{\sigma_1}^{\sigma_2}\langle \hat\ell(s),\tfrac1h\Delta_h\hat 
z(s)\rangle\ds.\label{convparam.engdissip1}
\end{align}
Now, since $s\mapsto\calI(\hat z(s))$ is uniformly continuous (as shown 
above), the first integral converges to $\calI(\hat 
z(\sigma_2))-\calI(\hat z(\sigma_1))$ with $h\to 0$. For the second integral, we 
have to distinguish the cases $s\in [0,S]\setminus G$, where we have 
$\vvm(\hat\ell(s),\hat z(s))=0$, and $s\in G$, where we can argue as follows: 
Since $\hat z\in W^{1,\infty}_\text{loc}(G;\calV)$, we find by the Dominated 
Convergence Theorem for every $K\Subset G$
\begin{align*}
\lim_{h\to0}\int_{(\sigma_1,\sigma_2)\cap K}\vvm(\hat\ell(s),\hat 
z(s))\Vert\tfrac1h\Delta_h\hat z(s)\Vert_\bbV\ds=\int_{(\sigma_1,\sigma_2)\cap 
K}\vvm(\hat\ell(s),\hat z(s))\Vert\hat z^\prime(s)\Vert_\bbV\ds .
\end{align*}
Furthermore, since we have $\hat z\in C([0,S];\calV)$, it follows that $\Delta_h 
\hat z(s)\to 0$ strongly in $\calV$ and uniformly in $s$, and since $\hat z\in 
AC^\infty([0,S];\calX)$, we infer by means of the results in Appendix 
\ref{app.bvac} that $\calR(\tfrac1h\Delta_h\hat z(s))\to \calR[\hat 
z^\prime](s)$ for almost every $s\in[0,S]$. Keeping in mind that 
$\calR[\hat z'(s)]\leq 1$ due to the normalization inequality, the 
Dominated Convergence 
Theorem implies that
\begin{align*}
\lim_{h\to0}\int_{\sigma_1}^{\sigma_2} (1+\lambda\Vert\Delta_h\hat 
z(s)\Vert_\bbV)\calR(\tfrac1h\Delta_h\hat 
z(s))\ds=\int_{\sigma_1}^{\sigma_2}\calR[\hat z^\prime](s)  \ds.
\end{align*}
Finally, for the term on the right hand side of \eqref{convparam.engdissip1}, we 
find
\begin{align}
\int_{\sigma_1}^{\sigma_2}\langle \hat\ell(s),&\tfrac1h\Delta_h\hat 
z(s)\rangle\ds=\frac1h\bigl( \int_{\sigma_1+h}^{\sigma_2+h}\langle 
\hat\ell(s-h)-\hat\ell(s),\hat z(s)\rangle+\langle\hat \ell(s),\hat 
z(s)\rangle-\langle\hat\ell(s-h),\hat z(s-h)\rangle\ds \bigr)\nonumber\\
&=-\int_{\sigma_1+h}^{\sigma_2+h}\langle 
\tfrac{\hat\ell(s)-\hat\ell(s-h)}{h},\hat 
z(s)\rangle\ds+\frac1h\int_{\sigma_2}^{\sigma_2+h}\langle\hat \ell(s),\hat 
z(s)\rangle\ds-\frac1h\int_{\sigma_1}^{\sigma_1+h}\langle\hat \ell(s),\hat 
z(s)\rangle\ds\nonumber\\
&=-\int_{\sigma_1}^{\sigma_2}\langle \tfrac{\hat\ell(s+h)-\hat\ell(s)}{h},\hat 
z(s+h)\rangle\ds+\frac1h\int_{\sigma_2}^{\sigma_2+h}\langle\hat \ell(s),\hat 
z(s)\rangle\ds-\frac1h\int_{\sigma_1}^{\sigma_1+h}\langle\hat \ell(s),\hat 
z(s)\rangle\ds.\label{convparam.engdissip2}
\end{align}
In order to show convergence, we apply (2) in Lemma \ref{convell:diffquotient} 
to $v=\hat \ell$ and (3) in the same Lemma to $v=\hat z$ and obtain for the 
first term on the right hand side of \eqref{convparam.engdissip2} the 
convergence
\begin{align*}
\lim_{h\to0}\int_{\sigma_1}^{\sigma_2}\langle 
\tfrac{\hat\ell(s+h)-\hat\ell(s)}{h},\hat 
z(s+h)\rangle\ds=\lim_{h\to0}\int_{\sigma_1}^{\sigma_2}\langle 
L_h\hat\ell(s),S_h\hat z(s)\rangle\ds=\int_{\sigma_1}^{\sigma_2}\langle 
\hat\ell^\prime(s),\hat z(s)\rangle\ds,
\end{align*}
while the second and third term converge to $\langle \hat\ell(\sigma_2),\hat 
z(\sigma_2)\rangle$ and $\langle \hat\ell(\sigma_1),\hat z(\sigma_1)\rangle$, 
respectively, for almost all $\sigma_1,\sigma_2$. In fact, since both $\hat 
\ell\in H^1(0,S;\calV^*)$  and $\hat z\in C([0,S],\calV)$ are continuous, the 
product $s\mapsto \langle \hat\ell(s),\hat 
z(s)\rangle$ is uniformly continuous on $[0,S]$, and we have convergence for all 
$\sigma_1, \sigma_2\in[0,S]$. 
Altogether, we can now pass to the limit in \eqref{convparam.engdissip1} and 
obtain the opposite estimate in \eqref{eq.energydissipidentitylimitpparam}, 
which is therefore valid as an identity.

 We can now proceed to show that the estimates \eqref{convparam.dissip1} and 
\eqref{convparam.dissip2} can be improved to equalities by standard arguments. 
Observe first that by arguments similar to the proof of the continuity of 
$s\mapsto\calI(\hat z(s))$, exploiting \eqref{est.convparam3} and Lemma 
\ref{lemma:convell} it follows that $\hat\calE_n(s,\hat z_n(s))\to \hat 
\calE(s,\hat z(s))$ for all $s\in [0,S]$.  Hence, from the energy dissipation 
balance written in the following way,
\begin{align*}
\lim_{n\to\infty}&\Bigl(\int_{0}^{\sigma}\calR[\hat z_n^\prime](s)  \ds+\int_{(0,\sigma)\cap 
G_n}\vvm(\hat\ell_n(s),\hat z_n(s))\Vert\hat 
z_n^\prime(s)\Vert_\bbV\ds+\hat\calE_n(\sigma,\hat z_n(\sigma))\Bigr)\\ 
&=\hat\calE(0,\hat z(0))+\int_{0}^{\sigma}\langle 
\hat\ell^\prime(s),\hat z(s)\rangle\ds\\
&=\int_{0}^{\sigma}\calR[\hat z^\prime](s)  \ds+\int_{(0,\sigma)\cap 
G}\vvm(\hat\ell(s),\hat z(s))\Vert\hat 
z^\prime(s)\Vert_\bbV\ds
+\hat\calE(\sigma,\hat z(\sigma)),
\end{align*}
we may conclude that in fact
\begin{align*}
\lim_{n\to\infty}\int_{0}^{\sigma}\calR[\hat z_n^\prime](s)  
\ds=\int_{0}^{\sigma}\calR[\hat z^\prime](s)  \ds 
\end{align*}
and
\begin{align*}
\lim_{n\to\infty}\int_{(0,\sigma)\cap 
G_n}\vvm(\hat\ell_n(s),\hat z_n(s))\Vert\hat z_n^\prime(s)\Vert_\bbV\ds=\int_{(0,\sigma)\cap 
G}\vvm(\hat\ell(s),\hat z(s))\Vert\hat z^\prime(s)\Vert_\bbV\ds
\end{align*}
are valid for all $\sigma\in[0,S]$. Now, writing $\int_{0}^{\sigma}\calR[\hat z_n^\prime](s)  \ds+\int_{(0,\sigma)\cap 
G}\vvm(\hat\ell_n(s),\hat z_n(s))\Vert\hat 
z_n^\prime(s)\Vert_\bbV\ds=\int_0^\sigma \left(1-\hat t_n^\prime(s)\right)\ds$, 
the above convergences yield the normalization condition 
\eqref{eq.normalized.sol}. 

Finally, exploiting the $\lambda$-convexity property 
\eqref{est.lambda-convex-energy} and keeping in mind that $\rmD\calI(\hat 
z(s))\in \calV^*$ for all $s$,  for every $s\in [0,S]$ we deduce that $\hat 
z_n(s)\rightarrow \hat z(s)$ strongly in $\calZ$, which is 
\eqref{est.convparam6}. 
\end{proof}

\section{The optimal control problem}
\label{sec:optcontr}
We now turn to the optimal control problem governed by 
\eqref{def:rate-indep-sys}. Our control variable is $\ell\in H^1(0,T;\calV^*)$ 
and the admissible set $M_{ad}$ consists of all normalized, $\vvp$-parametrized 
BV solutions of the system \eqref{def:rate-indep-sys} with data $z_0$ and 
$\ell$. To be more precise, we define 
\begin{align*}
M_{ad}:=&\left\lbrace (S,\hat t,\hat z,\ell)\in\R_+\times 
W^{1,\infty}(0,S)\times AC(0,S;\calX)\times H^1(0,T;\calV^*)\,|\right. \\ 
&\hspace{5cm}\left. (z_0,\ell) \text{ comply with }\eqref{eq.compatibledata}, 
\text{ and }(S,\hat t,\hat z)\in \mathcal{L}(z_0,\ell)  \right\rbrace.
\end{align*}
Then, the optimal control problem under consideration reads as follows:
\begin{align}
 \left. \begin{array}{l} \text{min}\quad J(S,\hat z,\ell):=j(\hat z(S)) 
+\alpha\Vert\ell\Vert_{H^1(0,T;\calV^\ast)}\\ \text{s.t. }\quad (S,\hat 
t, \hat z,\ell)\in M_{ad}.\end{array} \right\}\label{def.optcontr}
\end{align}
Herein, $\alpha>0$ is a fixed Tikhonov parameter and  $j:\calV\to \R$ is bounded 
from below and continuous, e.g. $j( z):=\Vert  z-z_\text{des}\Vert_\calV$ for a 
desired end state $z_\text{des}\in\calV$.

We now have the following existence result:
\begin{theorem}
Let $\alpha>0$ be a fixed Tikhonov parameter,  $z_0\in\calZ$ be chosen such that 
there exists $\ell\in H^1(0,T;\calV^*)$ such that $(z_0,\ell)$ complies with 
\eqref{eq.compatibledata} and let $j:\calV\to \R$ be bounded from below and 
continuous. Then, the optimal control problem \eqref{def.optcontr} has a 
globally optimal solution.
\end{theorem}
\begin{proof}
Since $j$ is prerequisited to be continuous  and bounded from below, we find 
that $I:=\inf \lbrace J(S,\hat z,\ell)\,|\, (S,\hat t, \hat z,\ell)\in 
M_{ad}\rbrace>-\infty$. We choose an infimizing sequence $((S_n,\hat 
t_n,\hat z_n,\ell_n))_{n\in\N}\subset M_{ad}$, i.e.
\begin{align*}
I=\lim_{n\to\infty}J(S_n,\hat z_n,\ell_n).
\end{align*}
Due to the boundedness assumption on $j$, we find that
\begin{align*}
R:=\sup_{n\in\N}\Vert \ell_n\Vert_{H^1(0,T;\calV^*)}<\infty,
\end{align*}
whence $((S_n,\hat t_n,\hat z_n))_{n\in\N}\subset M_{\Vert z_0\Vert_\calZ+R}$
with $M_\rho$ as in \eqref{eq.defMrho}. 
 According to Theorem \ref{thm.properties_sol_set} this set is compact. Thus, 
there exists a subsequence (not relabeled for simplicity) and limit elements 
$\ell_*\in H^1(0,T;\calV^*)$ and $(S_*,t_*, z_*)\in\mathcal{L}(z_0,\ell_*)$ such 
that $(z_0,\ell_*)$ comply with \eqref{eq.compatibledata} and we have in 
particular the convergences (cf. \eqref{est.convparam5})
\begin{align*}
\ell_n\rightharpoonup\ell_*\text{ in } H^1(0,T;\calV^*)\, \text{ and }\,  \hat z_n(S_n)\to z_*(S_*)\text{ in } \calV.
\end{align*}
Therefore, $(S_*, t_*, z_*,\ell_*)\in M_{ad}$, and since $j$ is assumed to be continuous, we infer that
\begin{align*}
I\leq J( S_*,z_*,\ell_*)\leq \liminf_{n\to\infty} \Bigl(j(\hat z_n(S_n))+\alpha\Vert \ell_n\Vert_{H^1(0,T;\calV^*)}\Bigr)=\lim_{n\to\infty}J(S_n,\hat z_n,\ell_n)=I,
\end{align*}
whence $(S_*, t_*, z_*,\ell_*)$ is indeed a minimizer of $J$ on the admissible set $M_{ad}$.
\end{proof}

\subsection*{Acknowledgments}
This research has been  funded by Deutsche 
Forschungsgemeinschaft 
(DFG) through the Priority Programme SPP 1962 Non-smooth and 
Complementarity-based Distributed Parameter Systems: Simulation and 
Hierarchical Optimization, Project P09 Optimal Control of Dissipative Solids: 
Viscosity Limits and Non-Smooth Algorithms.
\begin{appendix}

\section{Convergence of the load term}

The following convergence results are 
used Section \ref{subsec:compactness} to show convergence of the load term in 
the energy-dissipation balance \eqref{eq.energydissipidentitylimitpparam}.

\begin{lemma}\label{lemma:convell}

\begin{enumerate}
\item Let $\ell\in H^1(0,T;\calV^*)$ and $\hat t\in W^{1,\infty}(0,S)$ with 
$\hat{t}(s)\in[0,T]$ for all $s\in[0,S]$, $\hat t(0)=0$, and $\hat t(S)=T$. 
Then, it holds that $\ell\circ\hat t\in H^1(0,S;\calV^*)$ with 
\begin{align}
&(\ell\circ\hat t)^\prime(s)=\ell^\prime(\hat t(s))\hat t^\prime(s) \text{ f.a.a. }s\in[0,S], \label{chainrule:convell}\\
&\Vert\ell\circ \hat t\Vert_{L^2(0,S;\calV^*)}\leq C\Vert\ell\Vert_{H^1(0,T;\calV^*)},\label{convell:estL2}\\
&\Vert(\ell\circ \hat t)^\prime\Vert_{L^2(0,S;\calV^*)}\leq \Vert\ell^\prime\Vert_{L^2(0,T;\calV^*)}\Vert \hat t^\prime\Vert_{L^\infty(0,S)}^{\tfrac12},\label{convell:estellprime}
\end{align}
 for a constant $C>0$ depending on the space $\calV^*$ only. 
\item Let $(\ell_n)_{n\in\N}\subset H^1(0,T;\calV^\ast) $ and $(\hat 
t_n)_{n\in\N}\subset W^{1,\infty}(0,S)$ be sequences fulfilling 
\begin{align}
\hat{t}_n(0)=0,\quad \hat t_n(S)=T,\quad \text{f.a.a. }s\in[0,S]:\, 0\leq\hat 
t_n^\prime(s)\leq 1,\label{convell:assump1}
\end{align}
as well as
\begin{align}
\ell_n\rightharpoonup \ell\text{ in }H^1(0,T;\calV^\ast) \text{ and } \hat 
t_n\overset{*}\rightharpoonup\hat t\text{ in } W^{1,\infty}(0,S).\label{convell:assump2}
\end{align}
Set $\hat \ell_n:=\ell_n\circ\hat t_n$ for $n\in\N$ and 
$\hat\ell:=\ell\circ\hat 
t$. Then it holds that 
\begin{align*}
\hat\ell_n\rightharpoonup \hat\ell \text{ in }H^1(0,S;\calV^*),\\
\text{for all }s\in[0,S]: \hat\ell_n(s)\rightharpoonup\hat\ell(s) \text{ weakly in }\calV^\ast.
\end{align*}
\end{enumerate}
\end{lemma}

\begin{proof}
\underline{Proof of \textit{(1)}:} Let us first prove the chain rule 
\eqref{chainrule:convell} in analogy to the finite dimensional case. Let 
$s_0\in[0,S]$ be such that $\hat t$ is differentiable in $s_0$ and $\ell$ is 
differentiable in $\hat t(s_0)$. According to \cite[Thm. 1.4.35]{CH98}, this is 
the case almost everywhere in $[0,S]$. For $t, t_0\in[0,T]$, we define 
\begin{align*}
D(t,t_0):=\begin{cases}
\frac{\ell(t)-\ell(t_0)}{t-t_0},&\text{ if }t\neq t_0,\\
\ell^\prime(t_0),&\text{ if } t=t_0.
\end{cases}
\end{align*}
Then, for those points $t_0$ in which $\ell$ is differentiable, the map $D(\cdot,t_0):[0,T]\to\calV^*$ is norm-continuous in $t_0$. Therefore, it holds in $\calV^*$ that
\begin{align*}
\lim_{s\to s_0}\frac{\ell(\hat t(s))-\ell(\hat t(s_0))}{s-s_0}&=\lim_{s\to s_0}\Bigl(\frac{D(\hat t(s),\hat t(s_0))\cdot(\hat t(s)-\hat t(s_0))}{s-s_0}\Bigr)\\
&=\lim_{s\to s_0}\Bigl(D(\hat t(s),\hat t(s_0))\Bigr)\cdot\lim_{s\to s_0}\Bigl(\frac{\hat t(s)-\hat t(s_0)}{s-s_0}\Bigr)\\
&=D(\hat t(s_0),\hat t(s_0))\hat t^\prime(s_0)\\
&=\ell^\prime(\hat t(s_0))\hat t^\prime(s_0),
\end{align*}
where the second to last equation follows from the fact that $\hat t$ is continuous, and we infer \eqref{chainrule:convell}.

Now, since $H^1(0,T;\calV^*)\subset C([0,T],\calV^*)$ with a continuous embedding, it  holds that
\begin{align*}
\Vert \ell\Vert_{L^\infty(0,T;\calV^*)}=\Vert \ell\Vert_{C([0,T],\calV^*)}\leq C\Vert\ell\Vert_{H^1(0,T;\calV^*)}<\infty
\end{align*}
for the corresponding embedding constant $C>0$, whence 
\begin{align*}
\Vert \ell\circ\hat t\Vert_{L^2(0,S;\calV^*)}\leq \sqrt{S}\cdot \Vert \ell\Vert_{L^\infty(0,T;\calV^*)}\leq \sqrt{S}\cdot C\Vert\ell\Vert_{H^1(0,T;\calV^*)},
\end{align*}
which is \eqref{convell:estL2}. Furthermore, due to the chain rule \eqref{chainrule:convell}, it holds that
\begin{align*}
\int_0^S\Vert(\ell\circ\hat t)^\prime(s)\Vert_{\calV^*}^2\ds&=\int_0^S\Vert\ell^\prime(\hat t(s)\Vert^2_{\calV^*}(\hat t^\prime(s))^2\ds\\
&\leq\Vert\hat t^\prime\Vert_{L^\infty(0,S)}\int_0^S\Vert\ell^\prime(\hat t(s))\Vert^2_{\calV^*}|\hat t^\prime(s)|\ds\\
&=\Vert\hat t^\prime\Vert_{L^\infty(0,S)}\int_0^T\Vert\ell^\prime(t)\Vert^2_{\calV^*}\dt,
\end{align*}
yielding \eqref{convell:estellprime}, and we obtain that $\ell\circ\hat t \in 
H^1(0,S;\calV^\ast)$, which finishes the proof of \textit{(1)}.

\underline{Proof of \textit{(2)}:} Let $(\ell_n)_{n\in\N}\subset H^1(0,T;\calV^*)$ and $(\hat t_n)_{n\in\N}\subset W^{1,\infty}(0,S)$ be sequences fulfilling the assumptions \eqref{convell:assump1} and \eqref{convell:assump2}. We conclude by means of \textit{(1)} that $(\hat\ell_n)_{n\in\N}\subset H^1(0,S;\calV^\ast)$ and that $\sup_{n\in\N}\Vert \hat\ell_n\Vert_{H^1(0,S;\calV^\ast)}<\infty,$ so that
there is $\tilde{\ell}\in H^1(0,S;\calV^*)$ such that $\hat 
\ell_n\rightharpoonup\tilde\ell$ in $H^1(0,S;\calV^*)$ along a 
subsequence. In order to identify the limit, we first employ the embedding 
$H^1(0,T;\calV^*)\subset C^{0,\tfrac12}([0,T];\calV^*)$ 
(cf. \cite[Cor. 1.4.38]{CH98}) to infer for every $s\in[0,S]$ that
\begin{align*}
\Vert \ell_n(\hat t_n(s))-\ell_n(\hat t(s))\Vert_{\calV^*}\leq 
\Vert\ell_n\Vert_{C^{0,\tfrac12}([0,T];\calV^*)}|\hat t_n(s)-\hat t(s)|^{\tfrac12}\leq 
C|\hat t_n(s)-\hat t(s)|^{\tfrac12}\to 0.
\end{align*}
Moreover, the continuity of the above embedding also implies  that for every 
$t\in[0,T]$, the operator $\gamma_t:H^1(0,T;\calV^*)\to\calV^*$, 
$\gamma_t(\xi):=\xi(t)$ is well-defined and continuous, whence we further have 
that $\ell_n(\hat t(s))-\ell(\hat t(s))\rightharpoonup 0 \text{ in } \calV^*$ 
for every $s\in [0,S]$. Altogether we have the pointwise weak convergence
\begin{align*}
\hat\ell_n(s)-\hat\ell(s)=\ell_n(\hat t_n(s))-\ell_n(\hat t(s))+\ell_n(\hat 
t(s))-\ell(\hat t(s))\rightharpoonup 0 \text{ in }\calV^*.
\end{align*}
Thus, it holds that $\hat\ell_n\rightharpoonup\hat\ell=\tilde\ell\in 
H^1(0,S;\calV^*)$ along a subsequence. A standard proof by 
contradiction finally concludes the proof of Lemma \ref{lemma:convell}.
\end{proof}

\begin{lemma}\label{convell:diffquotient}
Let $0<h<h_0$. For $v\in H^1(0,S;\calV^*)$, consider the constant 
continuation to the interval 
$[0,S+h]$ and define the operator $L_h:H^1(0,S;\calV^\ast)\to L^2(0,S;\calV^*)$ 
by $L_h v(t):=\tfrac1h(v(t+h)-v(t))$. Then the following are true
\begin{enumerate}
\item $L_h$ is well defined, linear and continuous and for every $v\in 
H^1(0,S;\calV^*)$, it holds
\begin{align*}
\Vert L_hv\Vert_{L^2(0,S;\calV^*)}\leq \Vert v^\prime\Vert_{L^2(0,S;\calV^*)}.
\end{align*}
\item For all $v\in H^1(0,S;\calV^*)$, it holds that $L_hv\to v^\prime$ strongly 
in $L^2(0,S;\calV^*)$.
\end{enumerate}
For $v\in L^2(0,S;\calV)$, consider again the constant continuation to the 
interval $[0,S+h]$ and define the operator $S_h:L^2(0,S;\calV)\to 
L^2(0,S;\calV)$ by $S_hv(s):=v(s+h)$. Then
\begin{enumerate}\setcounter{enumi}{2}
\item $S_hv\to v$ strongly in $L^2(0,S;\calV)$.
\end{enumerate}
\end{lemma}
\begin{proof}
\underline{Proof of (1):} First assume that $v\in C^\infty([0,S];\calV^*)$, then 
we have
\begin{align*}
\Vert 
L_hv\Vert_{L^2(0,S;\calV^\ast)}^2&=\int_0^S\Vert\int_0^1v^\prime(t+sh)\ds\Vert_{
\calV^*}^2\dt\leq \int_0^S\int_0^1\Vert v^\prime(t+sh)\Vert_{\calV^*}^2\ds\dt\\
&=\int_0^1\int_0^S\Vert v^\prime(t+sh)\Vert_{\calV^*}^2\dt\ds=\int_0^1\Vert 
v^\prime(\cdot+sh)\Vert_{L^2(0,S;\calV^*)}^2\ds\\
&\leq \Vert v^\prime\Vert_{L^2(0,S;\calV^*)}^2.
\end{align*}
By density of $C^\infty([0,S];\calV^*)$ in $H^1(0,S;\calV^\ast)$ (cf. \cite[Satz 8.1.9]{Emmrich04}) 
, we obtain (1).\\
\underline{Proof of (2):} Again assume that $v\in C^\infty([0,S];\calV^*)$. Then 
it holds that
\begin{align*}
\Vert L_h v-v^\prime\Vert_{L^2(0,S;\calV^*)}=\int_0^S\Vert 
\tfrac{v(t+h)-v(t)}{h}-v^\prime(t)\Vert_{\calV^*}^2\dt\to 0,
\end{align*}
since the integrand is dominated by $C\Vert 
v^\prime\Vert^2_{L^\infty(0,S;\calV^*)}$ for a constant $C>0$. Now let $v\in 
H^1(0,S;\calV^*)$ and $\eta>0$ and choose $\overline{v}\in 
C^\infty(0,S;\calV^*)$ such that $\Vert v-\overline 
v\Vert_{H^1(0,S;\calV^*)}\leq \tfrac\eta3$ and $h_0>0$ so small that for all 
$0<h<h_0$ it holds that $\Vert L_h\overline v- \overline 
v^\prime\Vert_{L^2(0,S;\calV^\prime)}<\tfrac\eta3$. Then we have for all 
$0<h<h_0$
\begin{align*}
\Vert L_h v- v^\prime\Vert_{L^2(0,S;\calV^\prime)}&\leq \Vert L_h v- 
L_h\overline v\Vert_{L^2(0,S;\calV^\prime)}+\Vert L_h\overline v- \overline 
v^\prime\Vert_{L^2(0,S;\calV^\prime)}+\Vert \overline v^\prime- 
v^\prime\Vert_{L^2(0,S;\calV^\prime)}\\
&\leq 2\Vert \overline v^\prime- v^\prime\Vert_{L^2(0,S;\calV^\prime)}+\Vert 
L_h\overline v- \overline v^\prime\Vert_{L^2(0,S;\calV^\prime)}\\
&<\eta,
\end{align*}
which finishes the proof of (2). \\
\underline{Proof of (3):} 
Let again $v\in C^\infty([0,S];\calV)$, then it holds 
that 
\begin{align*}
\Vert S_hv-v\Vert_{L^2(0,S;\calV)}^2&=\int_0^S\Vert v(s+h)-v(s)\Vert_{\calV}^2\dd{s}\\
&\leq C\bigl(\int_0^{S-h}\Vert 
v(s+h)-v(s)\Vert_{\calV}^2\dd{s}+\int_{S-h}^S\Vert 
v(S)-v(s)\Vert_{\calV}^2\dd{s}\bigr)\\
&\leq C\bigl(\int_0^{S-h}\Vert v(s+h)-v(s)\Vert_{\calV}^2\dd{s}+h\Vert 
v\Vert_{L^\infty(0,S;\calV)}^2\bigr)\\
&\xrightarrow{h\to0}0,
\end{align*}
since the first integrand converges to $0$ and is bounded by $2\Vert 
v\Vert_{L^\infty(0,S;\calV)}^2$. Now, let $v\in L^2(0,S;\calV)$ and $\eta>0$. 
Since $C^\infty([0,S],\calV)$ is dense in $L^2(0,S;\calV)$ according to 
\cite[Remark 2.2.4]{GP06}, we can choose $\overline v\in C^\infty([0,S];\calV)$ such 
that $\Vert v-\overline v\Vert_{L^2(0,S;\calV)}^2<\frac\eta4$. We further find 
$h_1>0$ so small that for all $0<h<h_1$, it holds $\Vert 
S_h\overline v-\overline v\Vert_{L^2(0,S;\calV)}
^2<\frac\eta4$ as well as $h<\frac{\eta}{4\Vert v(S)-\overline v(S)\Vert_\calV^2}$. This 
implies for all $0<h<h_1$ that
\begin{align*}
\Vert S_hv-&v\Vert_{L^2(0,S;\calV)}^2\\
&\leq \Vert S_hv-S_h\overline v\Vert_{L^2(0,S;\calV)}^2+\Vert 
S_h\overline v-\overline v\Vert_{L^2(0,S;\calV)}^2+\Vert 
\overline v-v\Vert_{L^2(0,S;\calV)}^2\\
&\leq \Vert v-\overline v\Vert_{L^2(0,S;\calV)}^2+h\Vert v(S)-\overline v(S)\Vert_\calV^2 
+\Vert 
S_h\overline v-\overline v\Vert_{L^2(0,S;\calV)}^2+\Vert 
\overline v-v\Vert_{L^2(0,S;\calV)}^2\\
&<\eta,
\end{align*}
and (3) is proven.
\end{proof}


\section{Lower semicontinuity properties}

For the following Proposition we refer to \cite[Lemma 3.1]{MRS09}. 
\begin{proposition}
 \label{app_prop:lsc} 
Let $v_n,v\in L^\infty(0,S;\calV)$ with $v_n\overset{*}{\rightharpoonup}{v}$ in 
$L^\infty(0,S;\calV)$ and  $\delta_n,\delta\in L^1(0,S;[0,\infty))$ with 
$\liminf_{n\to\infty} \delta_n(s)\geq \delta(s)$ for almost all $s$. 
Then 
\begin{align}
 \label{app:eq20}
 \liminf_{n\to\infty}\int_0^S\norm{v_n(s)}_\calV \delta_n(s)\ds \geq 
\int_0^S 
\norm{v(s)}_\calV\delta(s)\ds.
\end{align}
\end{proposition}
%
The next lemma that we cite from 
\cite[Lemma 4.3]{MRS12VarConv} is closely related to the previous proposition:
\begin{lemma}
\label{app_prop:lsc2}
 Let $I\subset\R$ be a bounded interval and $f,g,f_n,g_n:I\to[0,\infty)$, $n\in 
\N$, measurable functions satisfying 
$
 \liminf_{n\to\infty} f_n(s)\geq f(s)$ for a.a.\ $s\in I$ and  
$g_n\rightharpoonup g$  weakly in $L^1(I)$. 
Then
\[
 \liminf_{n\to\infty}\int_I f_n(s)g_n(s)\ds\geq\int_I f(s)g(s)\ds\,.
\] 
\end{lemma}
A variant of the following lower semicontinuity property can be found in 
\cite[Lemma A.1]{KnRoZaPrep13}.  
\begin{lemma}\label{app_prop:lsc3}
Let $(\delta_n)_{n\in\N}\subset\R$ and $(\eta_n)_{n\in\N}\subset\calZ^*$ be  
sequences such that $\delta_n\searrow0$ and $\eta_n\rightharpoonup \eta$ weakly 
in $\calZ^*$. Then, it holds
\[\liminf_{n\to\infty}\calR_{\delta_n}^*(\eta_n)\geq\calR_\bbV^*(\eta).\]
\end{lemma}
\begin{proof}
Recall the definition $\calR_\delta(v):=\calR_\bbV(v)+\frac\delta2 \langle 
Av,v\rangle_{\calZ^*,\calZ}=\calR(v)+\frac12\langle\bbV 
v,v\rangle+\frac\delta2\langle 
Av,v\rangle_{\calZ^*,\calZ}=:\calR(v)+\calR_2(v)+\calR_{2,\delta}(v)$ for 
$v\in\Z$. Since both $\bbV\in\mathrm{Lin}(\calV,\calV^*)$ and 
$A\in\mathrm{Lin}(\calZ,\calZ^*)$ are supposed to be linear, continuous, 
symmetric and elliptic operators (cf. \eqref{eq.Mief0001}), they define 
equivalent norms $\Vert\cdot\Vert_\bbV$ and $\Vert\cdot\Vert_A$ on the spaces 
$\calV$ and $\calZ$, respectively. We denote the corresponding induced norms on 
the dual spaces $\calV^*$ and $\calZ^*$ by $\Vert 
\xi\Vert_{\bbV^{-1}}=\sqrt{\langle\xi,\bbV^{-1}\xi\rangle}$ and $\Vert \cdot 
\Vert_{A^{-1}}=\sqrt{\langle\xi,A^{-1}\xi\rangle}$, respectively. By standard 
arguments, we thus obtain for 
$\eta\in\calZ^*$
\begin{align*}
\calR^*(\eta)=\begin{cases}
0,&\text{ if }\eta\in\partial\calR(0),\\
\infty,&\text{ if }\eta\in \calZ^*\setminus\partial\calR(0),
\end{cases}\quad
\calR_2^*(\eta)=\begin{cases}
\frac12\Vert \eta\Vert_{\bbV^{-1}}^2, &\text{ if } \eta \in\calV^*,\\
\infty,&\text{ if }\eta\in\calZ^*\setminus\calV^*,
\end{cases}
\quad \calR_{2,\delta}^*(\eta)=\tfrac{1}{2\delta}\Vert\eta\Vert^2_{A^{-1}}.
\end{align*}
Furthermore, an application of \cite[Theorem 3.3.4.1]{IT79} yields
\begin{align*}
\calR_\delta^*(\eta)&=\inf\lbrace \calR^*(\eta_1)+\calR^*_2(\eta_2)+\calR^*_{2,\delta}(\eta_3)\, |\, \eta_1+\eta_2+\eta_3=\eta \rbrace\\
&=\inf\lbrace \calR^*_2(\eta_2)+\calR^*_{2,\delta}(\eta_3)\, |\, \eta-\eta_2-\eta_3\in\partial\calR(0) \rbrace\\
&=\min\lbrace \calR^*_2(\eta_2)+\calR^*_{2,\delta}(\eta_3)\, |\, \eta-\eta_2-\eta_3\in\partial\calR(0) \rbrace,
\end{align*}
since $\partial\calR(0)$ is a weakly closed subset of $\calV^*$.

Now, let $(\delta_n)_{n\in\N}\subset\R$ and $(\eta_n)_{n\in\N}\subset\calZ^*$ be sequences such that $\delta_n\searrow0$ and $\eta_n\rightharpoonup \eta$ weakly in $\calZ^*$. We denote the subsequence realizing the limit inferior with the same symbols for simplicity and choose $\eta_2^n,\eta_3^n\in\calZ^*$ such that $\eta_n-\eta_2^n-\eta_3^n\in\partial\calR(0)$ and $\calR_{\delta_n}^*(\eta_n)=\calR^*_2(\eta_2^n)+\calR^*_{2,\delta_n}(\eta^n_3)$. Assume that
\[\infty>\liminf_{n\to\infty}\calR_{\delta_n}^*(\eta_n)=\lim_{n\to\infty}\calR^*_2(\eta_2^n)+\calR^*_{2,\delta_n}(\eta^n_3)\geq 0.\]
This implies that $\sup_{n\in\N}\tfrac{1}{2\delta_n}\Vert\eta_3^n\Vert_{A^{-1}}^2<\infty$, whence $\eta_3^n\to 0$ strongly in $\calZ^*$, as well as $\sup_{n\in\N}\Vert\eta_2^n\Vert_{\bbV^{-1}}<\infty$. 
Thus, there exists a limit $\eta_2\in\calV^*$ such that $\eta_2^n\rightharpoonup\eta_2$ weakly in $\calV^*$ and $\eta_2^n\to\eta_2$ strongly in $\calZ^*$. Again, the closedness of $\partial\calR(0)$ yields $\eta_n-\eta_2^n-\eta_3^n\to \eta-\eta_2\in\partial\calR(0)$ strongly in $\calZ^*$. We can now use the weak lower semincontinuity of $\calR_2^*$ on $\calV^*$ to infer that
\begin{align*}
\lim_{n\to\infty}\calR_{\delta_n}^*(\eta_n)&=\lim_{n\to\infty}\Bigl(\calR^*_2(\eta_2^n)+\calR^*_{2,\delta_n}(\eta^n_3)\Bigr)\\
&\geq\liminf_{n\to\infty}\Bigl(\calR^*_2(\eta_2^n)+\calR^*_{2,\delta_n}(\eta^n_3)\Bigr)\\
&\geq\calR_2^*(\eta_2)\\
&\geq\inf\lbrace\calR_2^*(\tilde{\eta})\, |\, \eta-\tilde{\eta}\in\partial\calR(0)  \rbrace\\
&=\calR_\bbV^*(\eta),
\end{align*}
where in the last step we have again used \cite[Theorem 3.3.4.1]{IT79} to 
determine $\calR_\bbV^*$.   
\end{proof}

\section{Absolutely continuous functions and $BV$-functions}

\label{app.bvac} 
We follow \cite[Section 2.2]{MRS16}. Let $\calX$ be a Banach space and let 
$\calR:\calX\to\R$ be convex, lower semicontinuous, positively homogeneous of 
degree one and with \eqref{eq.Mief100}. For $1\leq p\leq \infty$, we define the 
set of $p$-absolutely continuous functions (related to $\calR$) as 
\begin{multline}
\label{eq.AC01}
 AC^p([a,b];\calX):=\big\{z:[a,b]\to\calX;\big. 
 \\
\big. \exists m\in L^p((a,b)),\, 
m\geq 0,\, \forall s_1<s_2\in [a,b]: \quad 
\calR(z(s_2)-z(s_1))\leq\int_{s_1}^{s_2}m(r)\dr\big\}\,.
\end{multline}
Observe that thanks to \eqref{eq.Mief100} this set coincides with the one 
defined with $\norm{\cdot}_\calX$ instead of $\calR$. Let $z\in 
AC^p([a,b];\calX)$. It is shown in \cite[Prop.\ 2.2]{MRS08-metricapproach}, 
\cite[Thm.\ 1.1.2]{AGS05} that for almost every $s\in [a,b]$ the limits
\begin{align*}
 \calR[z'](s):=\lim_{h\searrow0}\calR((z(s+h) - z(s))/h) = 
\lim_{h\searrow0}\calR((z(s) - z(s-h))/h)
\end{align*}
exist and are equal, that $\calR[z']\in L^p((a,b))$ and that $\calR[z']$ is the 
smallest function for which the integral estimate in \eqref{eq.AC01} is valid. 

Let further $\Var_\calR(z;[a,b])$ denote the $\calR$-variation of $z:[a,b]\to 
\calX$, i.e.\ 
\[
\Var_\calR(z;[a,b]):=\sup_{\text{partitions of $[a,b]$}} 
\sum_{i=1}^m \calR(z(s_i)-z(s_{i-1})).
\]
\begin{lemma}
 \label{app.lembv1}
For all $p\in (1,\infty]$ and 
$z\in AC^p([a,b];\calX)$ we have 
\begin{align}
 \label{eq.AC02}
 \Var_\calR(z,[a,b])=\int_a^b\calR[z'](s)\ds.
\end{align}
\end{lemma}
\begin{proof}
Since $\calR(z(s_i) - z(s_{i-1}))\leq \int_{s_1}^{s_2}\calR[z'](s)\ds$ 
the upper estimate $ \Var_\calR(z,[a,b])\leq\int_a^b\calR[z'](s)\ds$ is 
immediate. In order to obtain the opposite estimate we follow the 
ideas in the proof of Corollary 3.3.4 in \cite{AGS05}. Let $z\in 
AC^p([a,b],\calX)$ and choose a sequence of equidistant partitions $\pi_h$ of 
$[a,b]$ with fineness $h>0$. Let $z_h:[a,b]\to\calX$ be the linear interpolant 
of $z$ with respect to $\pi_h$. Since $z\in AC^p([a,b];\calX)$, the sequence 
$(z_h)_h$ converges pointwise in $\calX$ to $z$  as $h$ tends to zero.    
%
Moreover, direct calculations show that for all $h>0$ we have 
$\norm{\calR[z_h']}_{L^p((a,b))}\leq \norm{\calR[z']}_{L^p((a,b))}$. 
%
Hence, there exists $A\in L^p((a,b))$ and a subsequence such that 
$\calR[z_h']\rightharpoonup A$ weakly($*$) in $L^p((a,b))$. Now for almost 
every $s\in (a,b)$ we conclude $A(s)\geq \calR[z'](s)$. This can be seen as 
follows: Let $s_1<s_2\in (a,b)$ be arbitrary. Then 
\begin{align*}
 \calR(z(s_2) - z(s_1))\leq \liminf_{h\to 0} \calR(z_h(s_2) - z_h(s_1))\leq 
 \limsup_{h\to 0}\int_{s_1}^{s_2}\calR(z_h'(r))\dr =\int_{s_1}^{s_2} A(r)\dr.
\end{align*}
Hence, for almost every $s\in (a,b)$ we have 
\[
\calR[z'](s)=\lim_{h\searrow 0}\calR((z(s+h) - z(s))/h)\leq h^{-1}\int_s^{s+h} 
A(r)\dr =A(s).
\]
On the other hand, for all $h>0$ it holds $\Var_\calR(z;[a,b])\geq 
\int_a^b\calR(z_h'(r))\dr$. Taking the limit $h\to 0$ we ultimately arrive at 
$\Var_\calR(z;[a,b])\geq \int_a^b A(r)\dr\geq \int_a^b\calR[z'](r)\dr$. 
\end{proof}

\section{A combination of Helly's Theorem and the Ascoli-Arzel{\`a} Theorem}

The arguments are closely related to those in  \cite{MRS16,AGS05}.   
%
\begin{proposition} 
\label{prop.hellyarzasc-version2}
Let $\calZ$ be a reflexive Banach space, $\calV,\calX$ further Banach spaces 
such that  \eqref{eq.Mief000} is satisfied and assume that 
$\calR:\calX\to [0,\infty)$ complies with  \eqref{eq.Mief100}.   
\begin{itemize}
 \item[(a)] The set $AC^1([a,b];\calX)\cap L^\infty((a,b);\calZ)$ is 
contained in $C([a,b];\calV)$ and there exists $C>0$ such that for all $z\in 
AC^1([a,b];\calX)\cap L^\infty((a,b);\calZ)$ we have 
\[
\norm{z}_{C([a,b];\calV)}\leq  C(\norm{z}_{L^\infty((a,b);\calZ)} 
+\norm{\calR[z']}_{L^1((a,b))}).
\]
\item[(b)] Let $(z_n)_n\subset AC^\infty([a,b];\calX)\cap 
L^\infty((a,b);\calZ)$ be uniformly bounded in the sense that 
$A:=\sup_n\norm{z_n}_{L^\infty((a,b);\calZ)}<\infty$ and 
$B:=\sup_n\norm{\calR[z']}_{L^\infty((a,b))}<\infty$.

Then there exists $z\in AC^\infty([a,b];\calX)\cap 
L^\infty((a,b);\calZ)$ and a (not relabeled) subsequence $(z_n)_n$ such that 
\begin{align}
 z_n&\to z \text{ uniformly in }C([a,b];\calV),
\label{eq:arzelaascolihellymod02}
\\
\forall t\in [a,b]\quad  z_n(t)&\rightharpoonup z(t) \text{ weakly in }\calZ.
\label{eq:arzelaascolihellymod01}
\end{align}
\item[(c)] It is $L^\infty((a,b);\calZ)\cap 
C([a,b];\calV)\subset C_\text{weak}([a,b];\calZ)$.
\end{itemize}
\end{proposition}
\begin{proof}
In order to verify (a) let $z\in  AC^1([a,b];\calX)\cap L^\infty((a,b);\calZ)$. 
By the Ehrling Lemma, 
\cite{wloka87}, for every $\mu>0$ there exists $C_\mu>0$ such that for all 
$t,s\in [a,b]$ we have
 \begin{align*}
  \norm{z(t) - z(s)}_\calV &\leq \mu\norm{z(t) - z(s)}_\calZ + 
 C_\mu\norm{z(t) - z(s)}_\calX
\\
 &
 \leq 2\mu \norm{z}_{L^\infty(a,b;\calZ)} + \wt C_\mu\int_s^t \calR[ z'](r)\dr 
\,.
  \end{align*}
This implies that $z\in C([a,b];\calV)$ together with the norm estimate and (a) 
is proved.  

For (b) let $(z_n)_n\subset AC^\infty([a,b];\calX)\cap L^\infty((a,b);\calZ)$ 
as in 
part (b) of the Proposition. Again by Ehrling' Lemma for every $\mu>0$ there 
exists $C_\mu>0$ such that for all 
$t>s\in [a,b]$ and $n\in \N$ we have
 \begin{align*}
  \norm{z_n(t) - z_n(s)}_\calV &\leq \mu\norm{z_n(t) - z_n(s)}_\calZ + 
 C_\mu\norm{z_n(t) - z_n(s)}_\calX
\\
 &
 \leq 2\mu A + \wt C_\mu\int_s^t \calR[z_n'](r)\dr 
 \leq 2\mu A + C_\mu  B\abs{t-s}\,.
  \end{align*}
This implies the equicontinuity of the sequence $(z_n)_n$ in $C([a,b];\calV)$. 
Indeed, for $\varepsilon>0$ choose $\mu< \varepsilon/(4A)$ and 
$\delta <\varepsilon/(2B C_\mu)$. Then for all $n\in \N$, $s,t\in [a,b]$ with 
$\abs{s-t}<\delta$ we have $\norm{z_n(s) - z_n(t)}_\calV\leq \varepsilon$. 
Together with $z_n(t)\in \calK$ for all $t$ and $n$, by the classical version 
of the Arzel{\`a}-Ascoli Theorem, see e.g.\ 
\cite{Dieu69}, we obtain \eqref{eq:arzelaascolihellymod02} for a subsequence. 
After possibly extracting a further subsequence, the generalized version of 
Helly's Theorem, see e.g.\ \cite[Theorem 
3.2]{MaMi05} guarantees \eqref{eq:arzelaascolihellymod01}. 
By lower semicontinuity  we conclude that for every $s<t\in [a,b]$ 
\begin{align*}
 \calR(z(t) - z(s))\leq \liminf_n\calR(z_n(t) - z_n(s))\leq \int_s^t B\ds,
\end{align*}
whence $z\in AC^\infty([a,b];\calX)$.

Standard arguments finally imply that $L^\infty((a,b);\calZ)\cap 
C([a,b];\calV)\subset C_\text{weak}([a,b];\calZ)$.
\end{proof}

\section{Chain rules}

\begin{proposition}
 \label{prop.chainrules}
Let $z\in H^1((0,T);\calV)\cap L^\infty((0,T);\calZ)$ and 
$\rmD\calI(z(\cdot))\in L^\infty((0,T);\calV^*)$. Then for almost all $t$, 
the mapping $t\mapsto\calI(z(t))$ is differentiable and we have the identity 
\begin{align*}
 \frac{\rmd}{\rmd t}\calI(z(t))=\langle A z(t),\dot 
z(t)\rangle_{\calV^*,\calV} + \langle\rmD\calF(z(t)),\dot 
z(t)\rangle_{\calV^*,\calV}\,.
\end{align*}

Integrated version of the chain rule: 
Let $z\in W^{1,1}((0,T);\calV)\cap 
L^\infty((0,T);\calZ)$ with  $\rmD\calI(z(\cdot))\in L^\infty((0,T);\calV^*)$ 
and assume that $t\mapsto\calI(z(t))$ is continuous on $[0,T]$. Then for all 
 $t_1<t_2\in [0,T]$
\begin{align}
\label{eq.int.chain.rule}
 \calI(z(t_2))-\calI(z(t_1))=\int_{t_1}^{t_2}\langle \rmD\calI(z(r)),\dot 
z(t)\rangle_{\calV^*,\calV}\dr.
\end{align}
\end{proposition}
\begin{proof}
For the proof of the integrated version of the chain rule  
let $t_1<t_2\in [0,T)$ and $h_0>0$ such that $t_2+h_0\leq T$. Then for all 
$0<h\leq h_0$ the  $\lambda$-convexity estimate 
\eqref{est.lambda-convex-energy} implies 
\begin{multline*}
 h^{-1}\int_{t_1}^{t_2} \calI(z(t+h)) - \calI(z(t))\dt 
\\ \geq 
 \int_{t_1}^{t_2}\langle\rmD\calI(z(t)),h^{-1}(z(t+h) - z(t))\rangle\dt 
 -\tfrac{\lambda}{h}\int_{t_1}^{t_2}\norm{z(t+h) - z(t)}^2_\calV\dt,
\end{multline*}
where $\lambda>0$ depends on  $\norm{z}_{L^\infty(0,T;\calZ)}$. 
Thanks to the  continuity of $\calI(z(\cdot))$, for the left hand side 
we obtain $\lim_{h\to 0}h^{-1}\int_{t_1}^{t_2} \calI(z(t+h)) - \calI(z(t))\dt 
=\calI(z(t_2)) - \calI(z(t_1))$. Since $z\in W^{1,1}((0,T);\calV)$, on each 
$(t_1,t_2)\Subset(0,T)$ the difference quotients converge strongly 
in the following sense: $h^{-1}(z(\cdot+h) - z(\cdot))\to \dot z(\cdot)$ 
strongly in $L^1((t_1,t_2);\calV)$, \cite[Cor. 1.4.39]{CH98}. Thus 
the first integral on the 
right hand side converges to $\int_{t_1}^{t_2}\langle \rmD\calI(z(r)),\dot 
z(t)\rangle\dr$, while the second integral on the right hand side converges to 
zero. A similar argument for $h<0$ finally proves \eqref{eq.int.chain.rule}.
\end{proof}

%
%

\end{appendix}

\small

\end{document}